\documentclass[3p]{elsarticle}
\usepackage{longtable}
\usepackage{color}
\usepackage{hyperref}
\usepackage{amssymb}
\biboptions{sort}
\journal{Discrete Mathematics}

\newtheorem{theorem}{Theorem}
\newtheorem{lemma}{Lemma}
\newtheorem{proposition}{Proposition}
\newtheorem{corollary}{Corollary}
\newtheorem{remark}{Remark}

\newenvironment{proof}[1][\hspace{-1.0ex}]%
{\addvspace{1mm}{\bf Proof\hspace{1.0ex}{#1}.} }%
{\quad$\square$\par\addvspace{1mm}}

\begin{document}

\begin{frontmatter}

\title{The extended 1-perfect trades in small hypercubes}

\author[mymainaddress]{Denis S. Krotov}
\ead{krotov@math.nsc.ru}

\address[mymainaddress]{Sobolev Institute of Mathematics, pr. Akademika Koptyuga 4, Novosibirsk, Russia}

\begin{abstract}
An extended $1$-perfect trade is a pair $(T_0,T_1)$ of two disjoint binary distance-$4$ even-weight codes 
such that the set of words at distance $1$ from $T_0$ coincides with the set of words at distance $1$ from $T_1$.
Such trade is called primary if any pair of proper subsets of $T_0$ and $T_1$ is not a trade.
Using a computer-aided approach, 
 we classify nonequivalent primary extended $1$-perfect trades of length $10$,
constant-weight extended $1$-perfect trades of length $12$, 
and Steiner trades derived from them.
In particular, all Steiner trades with parameters $(5,6,12)$ are classified.
\end{abstract}


\end{frontmatter}

\section{Introduction}

Trades of different types are used to study, construct, and classify different kinds
of combinatorial objects (codes, designs, matrices, tables, etc.).
Trades are also studied independently, as some natural generalization of  objects of the corresponding type.
In the current paper, 
we classify small (extended) $1$-perfect binary trades.
The $1$-perfect trades are objects related to $1$-perfect codes 
(perfect codes with distance $3$).
% and can be considered as a generalization of such codes. 
A $1$-perfect code is a set $C$ of vertices of a graph such that $|C\cap B|=1$ for every ball $B$ of radius $1$.
A $1$-perfect trade is a pair $(T_0,T_1)$ of disjoint vertex sets of a graph such that 
$|T_0\cap B|=|T_1\cap B|\in\{0,1\}$ for every ball $B$ of radius $1$.
Formally, the $1$-perfect trades generalize the pairs of disjoint $1$-perfect codes;
in some cases, for every $1$-perfect code $C$ a disjoint mate $C'$ can be explicitly constructed,
e.g., $C'=C+100...0$ in a Hamming space.

In the theory of $1$-perfect codes, trades play an important role for the construction of codes with different properties
and the evaluation of their number.
There are not so many works where the class of $1$-perfect trades is studied independently 
\cite{Pot12:spectra}, \cite{VAK:2008}, \cite{VorKro:2014en};
however, the subsets of $1$-perfect codes called $i$-components, or switching components,
which are essentially a special kind of trade mates, are used in many constructions of such codes,
see the surveys in \cite{AvgKro:embed}, \cite{Hed:2008:survey}, \cite{Ost:2012:switching}, \cite{Rom:survey}, \cite{Sol:switchings}.

In the binary case, when the graph is the $n$-cube,
$1$-perfect codes exist if and only if $n+1$ is a power of two (see, e.g., \cite{MWS}), 
while $1$-perfect trades exist for every odd $n$ \cite{VAK:2008}
%($n$ cannot be even because every ball with center in $T_0$ 
%intersects with $(n+1)/2$ balls with centers in $T_1$). 
(it can be easily established by the local analysis that the size $n+1$ of a ball must be even if a $1$-perfect trade exists; 
see another explanation of this in Section~\ref{ss:cf}).
The last fact allows to consider recursive approaches in constructing and studying trades 
and to collect some experimental material for small $n$.
Note that one of the standard approaches to study $1$-perfect binary codes is to consider 
extended $1$-perfect codes,
taking into account a natural bijection between these two classes;
% obtained from $1$-perfect codes
% by appending the parity check bit to every codeword;
in this paper, we also favor the framework of extended $1$-perfect codes and, respectively,
extended $1$-perfect trades.
% As a partial result of the classification, we find that there is only one, up to equivalence, 
% pair of disjoint small Witt designs S$(5,6,12)$ and it cannot be continued to a triple of disjoint S$(5,6,12)$.

As a part of the study of properties of the classified trades, 
we consider the connection between $1$-perfect trades and 
Steiner trades. 
The Steiner trades are well known in the theory of combinatorial designs,
and there is a lot of literature on this topic, see the surveys \cite{HedKho:trades}, \cite{Street:trades}.
All vertices of a $1$-perfect trade at minimum distance from some fixed 
``non-trade'' vertex form a Steiner trade;
we consider the question which Steiner trades can be derived from  
$1$-perfect trades in such a way.

The paper is organized as follows. After the definitions (Section~\ref{s:def})
and the preliminary results (Section~\ref{s:aux}),
we classify primary extended $1$-perfect trades of length $8$.
The next three sections are devoted to computational results. 
In Section~\ref{s:10}, we classify 
the extended $1$-perfect trades of length $10$; 
in Section~\ref{s:12}, we classify 
the constant-weight
extended $1$-perfect trades of length $12$.
The lists of trades are given in tables, together with some additional information 
(automorphism group, dual space, derived Steiner trades, connection with the Witt design).
%%% In Section~\ref{s:uni}, we briefly mention the classification of so-called unitrades, which generalize trades.
%%%  The properties of the largest found unitrades in terms of ball packings and equitable partitions are discussed in separate Subsection~\ref{ss:96}.
In Section~\ref{s:constr}, we describe a concatenation construction of extended $1$-perfect trades,
showing that small trades can be utilized to construct trades of larger lengths; 
the construction also demonstrates the connection of the $1$-perfect trades with the latin trades, 
which are widely studied in the theory of latin squares \cite{Cav:rev} and latin hypercubes \cite{Potapov:2013:trade}.

In our computer-aided classification, we used general principles described in \cite{KO:alg}.
The programs were written in {\tt c++} (early versions used {\tt sage} \cite{sage});
{\tt nauty} \cite{nauty2014} was used to deal with automorphisms and isomorphisms.

\section{Definitions}\label{s:def}

\newcommand\HammingGraphalvedCube[1]{\ensuremath{\frac12\mathrm{H}(#1)}}
\newcommand\SteinerS[2]{\mathrm{S}(#1{-}1,#1,#2)}
\newcommand\HammingGraph[1]{\ensuremath{\mathrm{H}(#1)}}

We consider simple graphs $G= (\mathrm{V}(G),\mathrm{E}(G))$.
%we consider are simple (without loops and multiple edges), connected, and regular. 
The \emph{distance} $\mathrm{d}(x,y)$ between two vertices $x$ and $y$ of a connected graph 
is defined as the minimum length
of a path connecting $x$ and $y$.
Two sets $C$ and $S$ of vertices of a graph are \emph{equivalent} 
if there is an automorphism $\pi$ of the graph such that $\pi(S)=C$.
Two pairs $(C_0,C_1)$ and $(S_0,S_1)$ of sets of vertices of a graph are \emph{equivalent} 
if there is an automorphism $\pi$ of the graph such that either $\pi(S_0)=C_0$ and $\pi(S_1)=C_1$,
or $\pi(S_0)=C_1$ and $\pi(S_1)=C_0$. 
The \emph{automorphism group}, denoted $\mathrm{Aut}(C)$, of a vertex set $C$ is defined as its stabilizer in the graph automorphism group.

\subsection{Hamming graphs, halved $n$-cubes, and Johnson graphs}

The \emph{Hamming graph} \HammingGraph{n,q} (if $q=2$, the \emph{$n$-cube} \HammingGraph{n}) is a graph whose vertices are the words of length $n$ over the alphabet 
$\{0,\ldots,q-1\}$,
two words being adjacent if and only if they differ in exactly one position.
The \emph{weight} $\mathrm{wt}(x)$ of a word $x$ is the number of nonzeros in $x$.

The \emph{halved $n$-cube} \HammingGraphalvedCube{n} is a graph whose vertices are the even-weight 
(or odd-weight) binary words of length $n$,
two words being adjacent if and only if they differ in exactly two positions.

The \emph{Johnson graph} J$(n,w)$ is a graph whose vertices are the weight-$w$ 
binary words of length $n$,
two words being adjacent if and only if they differ in exactly two positions.

It is known (see, e.g., 
\cite[Th.~9.2.1]{Brouwer},
\cite[p.~265]{Brouwer},
\cite[Th.~9.1.2]{Brouwer}) that any automorphism of \HammingGraph{n}, \HammingGraphalvedCube{n\ge 5}, or J$(n,w)$ 
is a composition of a coordinate permutation and a translation to some binary word $x$, 
which is arbitrary in the case of \HammingGraph{n}, even-weight for \HammingGraphalvedCube{n}, the all-zero $0^n$ or all-one $1^n$ for J$(2w,w)$,
and only $0^n$ for J$(n,w)$, $n\ne 2w$. 
For a vertex set $C$, in addition to $\mathrm{Aut}(C)$, we will use the notation $\mathrm{Sym}(C)$, which denotes
the set of all coordinate permutations that stabilize $C$.

The \emph{Hamming distance} $\mathrm{d_H}(x,y)$ 
between two words $x$ and $y$ of the same length is the number of coordinates in which $x$ and $y$ differ, i.e., the distance in the corresponding Hamming graph.
Note that the graph distance in a Johnson graph or the halved $n$-cube is the half of the Hamming distance: $\mathrm{d}(x,y) = \mathrm{d_H}(x,y)/2$.

\subsection{$1$-perfect codes,
extended $1$-perfect codes,
Steiner systems, 
latin hypercubes}

A \emph{$1$-perfect code} is a set of vertices of \HammingGraph{n} such that every radius-$1$ ball contains exactly one codeword.

An \emph{extended $1$-perfect code} is a set of vertices of \HammingGraphalvedCube{n} such that every maximum clique contains exactly one codeword. 
Note that the maximum cliques in \HammingGraphalvedCube{n}, $n\ge 5$, are radius-$1$ spheres in \HammingGraph{n}.
%centered in odd-weight words. 
There is a one-to-one correspondence between the $1$-perfect codes in \HammingGraph{n-1} 
and the extended $1$-perfect codes in \HammingGraphalvedCube{n}: 
if, for some fixed $i\in\{1,\ldots,n\}$ (to be explicit, take, e.g., $i=n$), 
we delete the $i$th symbol from all codewords of an extended $1$-perfect code, 
then the resulting set will be a $1$-perfect code (inversely, 
the deleted symbol 
can be uniquely reconstructed as the modulo-$2$ sum of the other symbols).
Extended $1$-perfect codes in \HammingGraphalvedCube{n} 
and $1$-perfect codes in \HammingGraph{n-1}
exist if and only if $n$ is a power of $2$.

A \emph{Steiner $k$-tuple system} $\SteinerS{k}{n}$
is a set of vertices of J$(n,k)$, $n\ge 2k$, 
such that every maximum clique contains exactly one word from the set.
S$(2,3,n)$ and S$(3,4,n)$ are known as STS$(n)$ and SQS$(n)$,
Steiner triple systems and Steiner quadruple systems, respectively. 
Note that every maximum clique in J$(n,k)$ consists of all weight-$k$ binary words of length $n$ adjacent 
in \HammingGraph{n} with a given  word of weight $k-1$ 
(in the case $n=2k$, of weight $k-1$ or $k+1$).

There is a well-known connection between 
$1$-perfect codes and Steiner systems.
If a $1$-perfect 
(extended $1$-perfect) code of length $n$
contains the all-zero word $0^n$,
then its weight-$3$ (weight-$4$, respectively) codewords
 form a 
STS($n$) (SQS($n$)), which is called \emph{derived} from the code.
However, in contrast to the $1$-perfect codes, 
STS($n$)s 
exist for all $n\equiv 1,3\bmod 6$,
and
SQS($n$)s
exist for all $n\equiv 2,4\bmod 6$.
Clearly, such STS or SQS  cannot be derived if $n+1$ (respectively, $n$) is not a power of $2$.
The question if there exist non-derived STS($n$) (SQS($n$)) 
when $n+1$ (respectively, $n$) is a power of $2$ is a known open problem,
which is solved only for $n\le 16$ \cite{OstPot2007}, \cite[Satz 8.5]{Hergert:85}.

The parameters S$(5,6,12)$ play a special role in our classification.
A sextuple system S$(5,6,12)$ 
found in \cite{Carmichael:31} and \cite{Witt:37}
is unique up to equivalence \cite{Witt:37} and known as the
\emph{small Witt design}.

A \emph{latin hypercube} (if $n=3$, a \emph{latin square}) 
is a set of vertices of \HammingGraph{n,q} 
such that every maximum clique contains exactly one word from the set.
A maximum clique of \HammingGraph{n,q} consists of $q$ words differing in only one coordinate.
Often, a latin hypercube is imagined an $(n-1)$-dimensional table of size $q\times\ldots\times q$ filled by the values of the last, $n$th, coordinate.

\subsection{$1$-perfect, extended $1$-perfect, Steiner, and latin trades}
A \emph{$1$-perfect trade}
is a pair $(T_0,T_1)$ of disjoint nonempty sets 
of vertices of \HammingGraph{n} such that for every radius-$1$ ball $B$ it holds
\begin{equation}\label{eq:trade}
|T_0\cap B| = |T_1\cap B| \in \{0,1\}.
\end{equation}
An \emph{extended $1$-perfect trade} 
(an \emph{$\SteinerS{k}{n}$ trade}, a \emph{latin trade}) 
is a pair $(T_0,T_1)$ of disjoint nonempty sets of vertices of \HammingGraphalvedCube{n} 
(of J$(n,k)$ with $n\ge 2k$, of \HammingGraph{n,q}, respectively)
such that (\ref{eq:trade}) holds
for every maximum clique $B$. 

In what follows, \emph{trade} always means one of the four considered types of trades.
Each component $T_i$ of a trade $(T_0,T_1)$ is called a trade \emph{mate}.

\begin{remark}\rm
Often, trades are defined as unordered pairs $\{ T_0,T_1 \}$ \cite{HedKho:trades}.
In this paper, however, we find it convenient to use the ordered version of the definition.
 Also, it should be noted that there is a different terminology in the literature 
 (especially, in the works on latin trades, see, e.g., \cite{Cav:rev}),
where $(T_0,T_1)$ is called a bi\-trade, and each of $T_0$, $T_1$ is called a trade.
\end{remark}

The \emph{volume} of a trade $(T_0,T_1)$ is the cardinality of $T_0$ 
(equivalently, of $T_1$, 
as (\ref{eq:trade}) implies $|T_0|=|T_1|$).
The \emph{length} of $(T_0,T_1)$ means the length of words $T_0$ and $T_1$ consist of.
A trade $(T_0,T_1)$ is called \emph{primary} if it cannot be partitioned into two trades
$(T'_0,T'_1)$ and $(T''_0,T''_1)$, $T_0 = T'_0 \cup T''_0$, $T_1 = T'_1 \cup T''_1$.
The role of trades in the study of (extended) $1$-perfect codes, Steiner systems,
latin squares and hypercubes
is emphasized by the following fact: if $C_0$ and $C_1$ are different $1$-perfect codes,
extended $1$-perfect codes, Steiner $k$-tuple systems, or latin hypercubes 
with the same parameters,
then $(T_0,T_1)$, where $T_i = C_i\backslash C_{1-i}$, 
is a trade of the corresponding type.
In particular, we have $C_1 = T_1 \cap C_0 \backslash T_0$, 
i.e., with a trade we can get one object from the other.

Extended $1$-perfect trades and $1$-perfect trades 
are in the same one-to-one correspondence
as extended $1$-perfect codes and $1$-perfect codes. 
If $(T_0,T_1)$ is a $1$-perfect or extended $1$-perfect trade and $x\not\in T_0,T_1$ is a word at distance
$k$ from $T_0$, then the weight-$k$ words of $T_0+x$ and $T_1+x$ form an $\SteinerS{k}{n}$ trade, called
\emph{derived} from $(T_0,T_1)$.
Latin trades can be used
for the construction of trades of other types 
(see Section~\ref{s:constr}, the only place in this paper where the latin trades appear).

For trades consisting of words of weight $n/2$
(we call them \emph{constant-weight} trades), 
there is the following simple but remarkable correspondence.
\begin{proposition}\label{p:perf-stein}
  A pair $(T_0,T_1)$ of weight-$k$ binary words of length $2k$ is a $\SteinerS{k}{2k}$ trade 
  if and only if it is an extended $1$-perfect trade.
\end{proposition}
\begin{proof}
Let $G$ be the set of vertices of J$(2k,k)$, and also a subset of the vertex set of 
$\HammingGraphalvedCube{2k}$. 
For every maximum clique $B$ in $\HammingGraphalvedCube{2k}$, 
the set $B\cap G$ is either empty or 
a maximum clique in J$(2k,k)$. 
Moreover, every maximum clique in J$(2k,k)$ is represented 
in such a way.
Trivially, $|T_i \cap B | = |T_i \cap (B\cap G)|$  holds for every subset $T_i$ of $G$.
So, for such subsets, the definitions of $\SteinerS{k}{2k}$ trades and extended $1$-perfect trades
are equivalent.
\end{proof}

A tuple $(T_0,\ldots,T_{k-1})$ of $k\ge 2$ sets 
is called a \emph{$k$-way trade} if every two different sets from it form
a trade. The concepts defined above for the trades 
(length, volume, primary, derived) and Proposition~\ref{p:perf-stein} 
are naturally extended to 
$k$-way trades.

\subsection{Characteristic functions}\label{ss:cf}
Via the characteristic functions, the trades of the considered types 
can be represented as eigenfunctions of the corresponding graphs
with some special discrete restrictions. 
This allows the trades to be studied 
using approaches of algebraic combinatorics, 
see, e.g., \cite{KMP:16:trades}. 
An {\em eigenfunction} of a graph $G=(V,E)$ 
corresponding to an eigenvalue $\theta$ 
is a real-valued function $f$ over $V$ that is
not constantly zero and satisfies 
$\theta f(x)=\sum_{y:\{y,x\}\in E} f(y)$
for every $x$ in $V$.
The eigenvalues of the Hamming, Johnson, and halved $n$-cube 
graphs can be found,
e.g., in 
\cite[Th.~9.2.1]{Brouwer},
\cite[Th.~9.1.2]{Brouwer},
\cite[p.~264]{Brouwer}.

%e.g., in \cite[Th.~9.2.1, 9.1.2]{Brouwer}.

We define the {\em characteristic function} of a 
%($1$-perfect, extended $1$-perfect,  $\SteinerS{k}{n}$, or latin) 
 trade $(T_0,T_1)$ as the $\{0,1,-1\}$-function 
$\chi_{(T_0,T_1)} \stackrel{\scriptscriptstyle\mathrm{def}}= \chi_{T_0}-\chi_{T_1}$, 
where $\chi_{T}$ denotes the characteristic $\{0,1\}$-function of a vertex set $T$.

It is straightforward that the characteristic function of a 
$1$-perfect, extended $1$-perfect, $\SteinerS{k}{n}$, or latin trade
is an eigenfunction of the corresponding graph 
with the eigenvalue $-1$, $-n/2$, $-k$, $-n$, respectively.
Moreover, the characteristic function of an extended $1$-perfect trade,
considered as a function over the vertex set of \HammingGraph{n}, 
is an eigenfunction of \HammingGraph{n} with the eigenvalue $0$.

The graph \HammingGraph{n} has an eigenvalue $-1$ ($0$) if and only if $n$ is odd 
(even, respectively).
%; $-n/2$ is an eigenvalue of \HammingGraphalvedCube{n} if and only if $n$ is even.
This gives a necessary condition for the existence of $1$-perfect trades 
(extended $1$-perfect trades, respectively),
which turns out to be sufficient, see, e.g., \cite{VAK:2008}.
% For all ``admissible'' cases 
% (that is, when $n$ is odd for $1$-perfect trades 
% and $n$ is even for extended $1$-perfect trades), 
% trades exist, see, e.g., \cite{VAK:2008}.

For the other two types of trades, 
there are no restrictions on the parameters:
$-k$ is the smallest eigenvalue of all J$(n,k)$s, $n\ge 2k$;
and $-n$ is the smallest eigenvalue of all \HammingGraph{n,q}s.
For all parameters, $\SteinerS{k}{n}$ trades and latin trades exist, 
see, e.g., \cite{HedKho:trades}, \cite{Potapov:2013:trade}.

\subsection{The rank and the dual space}
The rank is one of the characteristics of nonlinear codes
that say how far a code is from being linear.
In the theory of $1$-perfect codes, 
the concept of rank plays an important role;
% we can say more about 
the structure of arbitrary 
$1$-perfect codes of small rank was studied in \cite{AvgHedSol:class}, \cite{HK:q-ary}.
In the current paper,
we use the affine rank of binary codes,
which is invariant under the automorphisms of the $n$-cube.

We consider the set of all binary 
words of length $n$ as a vector-space $F^n$ over the finite field
of order $2$, with coordinate-wise modulo-$2$ addition and multiplication by a constant.

Let $C$ be a set of binary words of length $n$ (a code).
A binary word $x=(x_0,\ldots,x_{n-1})$ is said to be 
\emph{orthogonal} (\emph{antiorthogonal}) to $C$ if
$x_0c_0+\ldots +x_{n-1}c_{n-1}\equiv 0 \bmod 2$ 
(respectively, $\mbox{} \equiv 1 \bmod 2$) for all
$(c_0,\ldots,c_{n-1})$ from $C$.
The set of all 
binary words that are
orthogonal or antiorthogonal to $C$
is called the \emph{dual space} of $C$ (in affine sense) and denoted $C^\perp$.
\newcommand\rank{\mathrm{rank}}
The (affine) \emph{rank}, $\rank(C)$, of a code $C\subseteq F^n$,
is the minimum dimension
of an affine subspace of $F^n$ including $C$.
It is straightforward that $\rank(C)+\rank(C^\perp)=n$.

\section{Preliminary results}\label{s:aux}
The following two facts, Lemma~\ref{p:antipod} and its corollary, are well known.

\begin{lemma}\label{p:antipod}
Let $\phi$ be an eigenfunction of \HammingGraph{n} with the eigenvalue $n-2i$, $i\in\{0,1,\ldots,n\}$.
For every vertex $x$ of \HammingGraph{n}, it holds $\phi(x+1^n) = (-1)^i\phi(x)$,
where $1^n$ is the all-one word.
\end{lemma}
\begin{proof}
  It is straightforward that for every vertex $y = (y_1,\ldots,y_n)$, the function 
  $\phi_y(x_1,\ldots,x_n) \stackrel{\scriptscriptstyle\mathrm{def}}= (-1)^{y_1x_1+\ldots+y_nx_n}$ is an eigenfunction corresponding to the eigenvalue
  $n-2\mathrm{wt}(y)$ 
  (indeed, every vertex has $n-\mathrm{wt}(y)$ neighbors with the same values of $\phi_y$ and $\mathrm{wt}(y)$ neighbors with the opposite values). 
  All these $2^n$ functions are mutually orthogonal: for $y\ne z$,
  $$ \sum_{x\in \{0,1\}^n} \phi_y(x) \phi_z(x) 
  = \sum_x (-1)^{y_1x_1+\ldots+y_nx_n}(-1)^{z_1x_1+\ldots+z_nx_n} = \sum_x (-1)^{(y_1+z_1)x_1+\ldots+(y_n+z_n)x_n} = 0.$$
  Hence, they form a basis of the space of real-valued functions on the vertex set of \HammingGraph{n}.
  Consequently, any eigenfunction $\phi$ with the eigenvalue $n-2i$ is a linear combination of
  $\phi_y$ with $\mathrm{wt}(y)=i$. 
  So, it is sufficient to check that the statement of the lemma holds for every such $\phi_y$:\\[1mm]
  \mbox{}\hfill$\phi_y(x+1^n) = (-1)^{y_1(x_1+1)+\ldots+y_n(x_n+1)} = (-1)^{y_1+\ldots+y_n} (-1)^{y_1x_1+\ldots+y_nx_n} = (-1)^{i}\phi_y(x).$\hfill\mbox{}
\end{proof} 
 Since the characteristic function of any $1$-perfect or extended $1$-perfect trade
 (in particular, any $\SteinerS{k}{2k}$ trade, by Proposition~\ref{p:perf-stein})
 is an eigenfunction of the Hamming graph 
 corresponding to the eigenvalue $-1$ or $0$ (see Section~\ref{ss:cf}),
 we conclude from Lemma~\ref{p:antipod}
that such a trade
is self-complementary, in the following sense.
\begin{corollary}\label{c:antipod}
  Let $(T_0,T_1)$ be a $1$-perfect trade in \HammingGraph{n-1}
  or an extended $1$-perfect trade in \HammingGraphalvedCube{n}.
  Denote $\overline T_j \stackrel{\scriptscriptstyle\mathrm{def}}= T_j + 11...1 = \{x+11...1 \mid x\in T_j\}$, $j=0,1$.
  
  (i) If $n\equiv 0 \bmod 4$,  
  then $T_0 = \overline T_0 $ and  $T_1 = \overline T_1$. 
  
  (ii) If $n\equiv 2 \bmod 4$,  then $T_0 = \overline T_1$.
  \end{corollary}

We see that there is an essential difference between the two cases, (i) and (ii). 
In case (ii), $T_1$ is uniquely determined from $T_0$. 
They are complementary to each other and thus equivalent.
In case (i), each element of a trade  $(T_0,T_1)$ is self-complementary, 
but $T_0$ does not uniquely define $T_1$ and vice versa;
$T_0$ and $T_1$ can be nonequivalent in this case (see Section~\ref{ss:descr} for examples).
In fact, for every $n\equiv 1 \bmod 4$, 
there exists a $3$-way $1$-perfect trade \cite{VAK:2008}.
For $k>3$, $k$-way trades can also exist.

The next lemma (a partial case of the construction in Section~\ref{s:constr}) guarantees that the number of different 
(and, in fact, the number of nonequivalent) $1$-perfect trades 
is monotonic in $n$. Here and everywhere in the paper, for a symbol $\sigma$ and a set $T$ of words,
by $T\sigma$ we denote the set of words obtained by appending $\sigma$ to the words of $T$ (for example, in the notation $T_000$ below, this rule is applied twice) .
\begin{lemma}[\cite{VAK:2008}]\label{p:extend}
Let $(T_0,T_1)$ be a (extended) $1$-perfect trade in \HammingGraph{n}. 
Then $(T_000\cup T_111,T_011\cup T_100)$ is a (extended) $1$-perfect trade in \HammingGraph{n+2}.
\end{lemma}

The following easy-to-prove fact plays a crucial role in our classification algorithm.

\begin{proposition}[{\cite[Theorem~1]{KMP:16:trades}}]\label{p:subgraph}
Suppose  $T_0$, $T_1$ are disjoint vertex sets of \HammingGraphalvedCube{x} (or J$(n,k)$). The pair $(T_0,T_1)$
is an extended $1$-perfect trade (an $\SteinerS{k}{n}$ trade, respectively) if and only if the subgraph induced by 
$T_0\cup T_1$ is bipartite with parts $T_0$, $T_1$ and regular of degree $n/2$ (of degree $k$, respectively).
\end{proposition}

\begin{corollary}\label{c:prime}
An extended $1$-perfect trade $(T_0,T_1)$ in \HammingGraphalvedCube{n} is primary 
if and only if the corresponding induced subgraph is connected. 
The same is true for the $\SteinerS{k}{n}$ trades.
\end{corollary}

The next lemma is convenient for representing the dual space of 
a primary trade by a basis. 
Note that, for a primary extended $1$-perfect or primary $\SteinerS{k}{n}$ trade $(T_0,T_1)$,
the union  $T_0\cup T_1$ induces a connected subgraph of $\HammingGraphalvedCube{n}$.

\begin{lemma}
Let $C$ be a vertex set of $\HammingGraphalvedCube{n}$ 
such that the induced subgraph is connected.
% (for example, $C=T_0\cup T_1$, where $(T_0,T_1)$ 
% is a primary extended $1$-perfect or primary $\SteinerS{k}{n}$ trade).
Then the dual space of $C$ is closed with respect to the coordinate-wise multiplication.
\end{lemma}
\begin{proof}
 Let $x,y\in C^\perp$, and let $c$ and $d$ be codewords of $C$
 differing in exactly two coordinates, $i$ and $j$. 
 Denote by $z$ the coordinate-wise multiplication of $x$ and $y$.
 Since $x$ is orthogonal 
 or antiorthogonal to $\{c,d\}$,
 we have $x_i = x_j$. 
 Similarly, $y_i=y_j$. 
 It follows that
 $z_i=z_j$, 
 and we see that $z$ 
 is orthogonal 
 or antiorthogonal to $\{c,d\}$.
 From the connectivity, we get that $z$ 
 is orthogonal 
 or antiorthogonal to $C$. 
 \end{proof}
Any space closed with respect to the coordinate-wise multiplication
can be represented by the \emph{standard} basis whose elements have
mutually disjoint sets of non-zero coordinates.

\section{Extended $1$-perfect trades in \HammingGraphalvedCube{8}}\label{s:8}

As noted in Section~\ref{ss:cf}, extended $1$-perfect trades in \HammingGraphalvedCube{n} exist only if $n$ is even.
An example of a trade can be constructed recursively, starting with the trivial trade $(\{00\},\{11\})$
and applying the construction in Lemma~\ref{p:extend}.
Before we start our classification for $n=8$,
we note that this is the first case when nonequivalent extended $1$-perfect trades exist.
Indeed, the case $n=2$ is trivial,
In the case $n=4$, 
by Corollary~\ref{c:antipod}, 
every trade mate includes a self-complementary pair of vertices; 
since every maximum clique intersects with such pair, the volume is $2$.
If $(T_0,T_1)$ is an extended $1$-perfect trade in \HammingGraphalvedCube{6},
then we can assume without loss of generality that $000000\in T_0$ and,
in accordance with Proposition~\ref{p:subgraph}, $000011$, $001100$, $110000 \in T_1$;
then, by Corollary~\ref{c:antipod}, $T_0=\{000000,111100,110011,001111\}= T_1+1^6$.

Now consider three extended $1$-perfect codes of length $8$,
$C_0=\langle 00001111, 00110011, 01010101, 11111111 \rangle$,
$C_1=\langle 10000111, 00110011, 01010101, 11111111 \rangle$,
$C_2=\langle 00001111, 00110101, 01010110, 11111111 \rangle$,
where $\langle \dots \rangle$ denotes the linear span over the finite field of order $2$. 
It is not difficult to check that 
$(C_0\backslash C_1,C_1\backslash C_0)$,
$(C_0\backslash C_2,C_2\backslash C_0)$, and
$(C_1\backslash C_2,C_2\backslash C_1)$ are constant-weight
extended $1$-perfect trades of volume $8$, $12$, and $14$, respectively.
As we see from the following theorem, all nonequivalent primary extended $1$-perfect trades
are exhausted by these three constant-weight trades and two trades of volume $16$ (each
consisting of two extended $1$-perfect codes).

\begin{theorem}\label{th:8}
There are only $5$ nonequivalent primary extended $1$-perfect trades in \HammingGraphalvedCube{8},
of volume $8$, $12$, $14$, $16$ and $16$, respectively.
\end{theorem}
\begin{proof}
Step 1.  Without loss of generality, we can assume that 
$000000000 \in T_0$ and $v_1=11000000$, $v_2=00110000$, $v_3=00001100$, $v_4=00000011 \in T_1$.
By Corollary~\ref{c:antipod}, we also have
$1^8 \in T_0$ and $v_1+1^8$, $v_2+1^8$, $v_3+1^8$, $v_4+1^8 \in T_1$.

Step 2.
Now consider the word $v_1=11000000$ from $T_1$. 
We know one of its neighbors in $T_0$, $0^8$. 
There are $15$ ways to choose the set $\{x,y,z\}$ of three words adjacent to $V_1$ and not adjacent mutually and with $0^8$
(indeed, $v_1+0^8$, $v_1+x$, $v_1+y$, $v_1+z$ must be weight-$2$ words with mutually disjoint sets of ones).
Without loss of generality, it suffices to consider only three of them (each of the other cases can be obtained from these three
by applying one of $4!\cdot 2^4$ coordinate permutations stabilizing the collection of words chosen at Step 1):

a) $x=11110000$, $y=11001100$, $z=11000011$,

b) $x=11110000$, $y=11001010$, $z=11000101$,

c) $x=11101000$, $y=11010010$, $z=11000101$.

In each of the cases, $x+1^8$, $y+1^8$, and $z+1^8$ also belong to $T_0$.

(a) In this case, no more words can be added to $T_1$ or $T_2$ 
as the subgraph generated by the set of chosen $16$ vertices satisfies the condition 
of Proposition~\ref{p:subgraph}.

(b) We already know all four $T_0$-neighbors of $v_2=00110000$: $0^8$, $x$, $y+1^8$, $z+1^8$. 
Consider $v_3=00001100$. We know two its neighbors from $T_0$: $0^8$ and $x=00001111$.
The words $11001100$ and $00111100$ are adjacent to $y$ and $y+1^8$, respectively,
and hence cannot belong to $T_0$. Consequently, either $10101100$, $01011100\in T_0$,
or $10011100$, $01011100\in T_0$. Without loss of generality we consider the former case 
(all previously chosen words have coinciding values in the first two coordinates).
Then, $01010011$ and $10100011$ are also in $T_0$ and in the neighborhood of $v_4$.

Step 3 (case (b)).
For the word $x$ from $T_0$, we know its four neighbors from $T_1$: $v_1$, $v_2$, $v_3+1^8$, and $v_4+1^8$.
Consider $y=11001010\in T_0$. We have $11000000$, $11001111\in T_1$; 
the other two neighbors of $y$ from $T_1$ are

(i) $10101010$ and $01011010$ or

(ii) $10011010$ and $01101010$

(note that the third subcase, $00001010$ and $11111010$, is not feasible).

In the first subcase, including also the complements, we get
\begin{eqnarray*}
T_0 &\supseteq &\{ 00000000, 11111111, 11110000, 00001111, 11001010, 00110101,
\\&& \phantom{\{} 
11000101, 00111010, 10101100, 01010011, 01011100, 10100011\} ,
\\
T_1 &\supseteq& \{ 11000000, 00110000, 00001100, 00000011, 10101010, 01010101,
\\&& \phantom{\{} 
01011010, 10100101, 00111111, 11001111, 11110011, 11111100 \}.
\end{eqnarray*}
We see that the words chosen already form a trade; 
so,  $T_0$ and $T_1$ do not contain more vertices.

In the second subcase, including also the complements, we get
\begin{eqnarray*}
T_0 &\supseteq &\{ 00000000, 11111111, 11110000, 00001111, 
11001010, 00110101, 11000101, 00111010, \\&& \phantom{\{}
\underline{10101100}, \underline{01010011}, \underline{01011100}, \underline{10100011}\} ,
\\
T_1 &\supseteq& \{ 11000000, 00110000, 00001100, 00000011, 00111111, 11001111, 11110011, 11111100, \\&&
\phantom{\{} \underline{10011010}, \underline{01100101}, \underline{01101010}, \underline{10010101} \}.
\end{eqnarray*}
Consider $z={10101100}$. 
We know $00001100$, $11111100\in T_1$. 
The other two neighbors of $z$ in $T_1$ can be
$10100000$, $10101111$ (this way is not feasible as $11000000\in T_1$),
$10101010$, $10100101$ (not feasible as $10011010 \in T_1$),
or $10101001$, $10100110$, the only feasible way.

Similarly, considering $10011010\in T_1$, we find that
$10011001$, $10010110\in T_0$. Including also the complements, we have
two extended $1$-perfect codes:
\begin{eqnarray}
T_0&=&  \{ 00000000, 11111111, 11110000, 00001111, 
11001010, 00110101, 11000101, 00111010, \nonumber\\ &&
\phantom{\{}
{10101100}, {01010011}, {01011100}, {10100011}, 
10011001, 10010110, 01100110, 01101001\},  \nonumber
\\
T_1 &=& \{ 11000000, 00110000, 00001100, 00000011, 
00111111, 11001111, 11110011, 11111100, \label{eq:pp1}\\ &&
\phantom{\{}
{10011010}, {01100101}, {01101010}, 
{10010101},10101001, 10100110, 01010110, 01011001\nonumber\}.
\end{eqnarray}

(c) $00000000$, $11111111$,  $11101000$, $11010010$, $11000101$, $00010111$, $00101101$, $00111010\in T_0$.

Consider the word $v_2=00110000\in T_1$ and its possible neighbors from $T_0$.
We know $00000000$, $00111010\in T_0$. 
We see that
$11110000\not\in T_0$ (because $11101000\in T_0$).
The remaining words are
$10110100$, $01110100$, $10110001$, $01110001$.
Without loss of generality, 
$10110100$, $01110001\in T_0$.
The complements $01001011$ and $10001110$ are also in $T_0$.

Next, consider the neighborhood of $v_3=00001100\in T_1$.
We know $00000000$, $00101101$, $10001110\in T_0$; so,
we deduce that $01011100\in T_0$.
The complement $10100011$ is the fourth neighbor of $v_4$ in $T_0$.

Now we know that
\begin{eqnarray*}
T_0&\supseteq& \{ 00000000, 11111111,  11101000, 11010010, 11000101, 00010111, 00101101, 
\\&& \phantom{\{} 
00111010, 10110100, 01110001, 01001011, 10001110,
01011100, 10100011\}.
\end{eqnarray*}

Step 3 (case (c)). Consider the neighborhood of $x=11101000$. It contains $v_1$ and $v_4+1^8$ from $T_1$. 
There are two ways to choose the other two trade neighbors of $x$.

(i) $10101010$ and $01101001$ (and their complements $01010101$ and $10010110$) are in $T_1$.

Then, $y=11010010$ is adjacent to $v_1$, $v_3+1^8$, and $10010110$ from $T_1$. So, the fourth trade neighbor is $01011010$; the complement is $10100101$.
At this stage, we have 
\begin{eqnarray*}
T_1&\supseteq& \{ 11000000, 00110000, 00001100, 
00000011, 00111111, 11001111, 11110011,  
\\&& \phantom{\{} 
11111100, 10101010, 01101001, 
01010101, 10010110, 01011010, 10100101 \}
\end{eqnarray*}
and no more words can be added to $T_0$ or $T_1$. Then, $(T_0,T_1)$ is a trade of volume $14$.

(ii) $01101010$ and $10101001$ (and their complements $10010101$ and $01010110$)  are in $T_1$.

Then, $y=11010010$ is adjacent to $v_1$, $v_3+1^8$, and $01010110$ from $T_1$.  
The fourth $T_1$-neighbor of $y$ must be $10010110$; the complement is $01101001$.
Next, $10110100$ from $T_0$ has neighbors $v_2$, $v_4+1^8$, and $10010101$ from $T_1$.
The fourth neighbor from $T_1$ must be $10100110$; the complement is $01011001$.
Now, we know $16$ words of $T_1$. Since $16$ is the cardinality of an extended $1$-perfect code,
no more words can be added. Two more words should be found in $T_0$; 
it is not difficult to see that the only way
is $01100110$, $10011001$. 
We have two disjoint extended $1$-perfect codes:
\begin{eqnarray}
 T_0 &=& \{
 00000000, 11111111, 11101000, 11010010, 
 11000101, 00010111, 00101101, 00111010, \nonumber \\
&& \phantom{\{}
 10110100, 01110001, 01001011, 10001110,
 01011100, 10100011, 01100110, 10011001 \nonumber \},
\\
T_1&=&\{
 11000000, 00110000, 00001100, 00000011, 
 00111111, 11001111, 11110011, 11111100, \label{eq:pp2} \\
&& \phantom{\{}
 01101010, 10101001, 10010101, 01010110, 
 10011010, 01100101,  10100110, 01011001 \nonumber \}.
\end{eqnarray}

It remains to show that the solutions (\ref{eq:pp1}) and (\ref{eq:pp2}) are nonequivalent.
We count the number of words orthogonal to $T_0 \cup T_1$ 
(it is easy to see that this number is invariant among equivalent codes).
The unique, up to equivalence, extended $1$-perfect code containing $0^8$ is self-dual, 
that is, the set of all orthogonal words is the code itself.
So, every orthogonal word belongs to $T_0$; 
for each word of $T_0$, it is easy to check if it is orthogonal to $T_1$.
As a result, for (\ref{eq:pp1}) and (\ref{eq:pp2}), the dual spaces are $\{0^n, 1^n, 00001111, 11110000\}$
and $\{0^n, 1^n\}$ respectively, which certifies that the corresponding trades are nonequivalent.
\end{proof}

\section{Extended $1$-perfect trades in \HammingGraphalvedCube{10}}\label{s:10}

In this section, we describe a computer-aided classification 
of extended $1$-perfect trades of length $10$.
The computation took a few seconds on a modern PC.

\subsection{Algorithm}\label{aa:alg}
The algorithm described below is similar to the one used in the proof of Theorem~\ref{th:8}.
We omit the details concerning some natural improvement and show only the general approach.
Essentially, it is the breadth-first search of a bipartite $5$-regular induced subgraph of \HammingGraph{10}
that takes into account the complementarity.
Below, we consider $T_0$ and $T_1$ as lists of words, 
whose contents change during the run of the algorithm.

At step 1, we assume that 
$T_0$ contains $0^{10}$ and  
$T_1$ contains $v_1=1100000000$,  
$v_2=0011000000$,
$v_3=0000110000$,
$v_4=0000001100$, and
$v_5=0000000011$; utilizing Corollary~\ref{c:antipod}, we add $1^{10}$ to $T_1$ and $v_i+1^{10}$ to $T_0$, $i=1,\ldots,5$.
Since any trade is equivalent to one with these words, these twelve words will not be changed during the search.

At step 2, for $i$ from $1$, \ldots, $5$, we choose  
the lexicographically first collection of $5$ mutually non-adjacent words 
in the neighborhood of $v_i$ that are not adjacent to any word of $T_0$.
This implies that any word of $T_0$ (say, $0^{10}$) adjacent to $v_i$ is automatically chosen.
The other chosen words are ``new'', and we include them to $T_0$, and their complements to $T_1$.
If, for some $i$, there is no such collection of $5$ words, we return to $v_{i-1}$ 
and choose the next lexicographical alternative for it (if there is no such alternative, return to $v_{i-2}$, and so on).
When the $5$ neighbors are chosen for all $v_i$, $i=1,\ldots,5$, we come to the next step.

At step 3, for each  word of $T_0$ added at the previous step, we find $5$ mutually 
non-adjacent neighbors that are not adjacent to any word of $T_1$,
 add the chosen words that are new to $T_1$, 
 and add the complements to $T_0$. 
Again, after considering all possibilities for a given vertex, 
we roll back to the previous vertex,
which is at this or the previous step,
and choose the next alternative for it.

Similarly, step 4, step 5, and so on.

If, at some step, we find that each included words in $T_0$ ($T_1$) already has $5$
neighbors in $T_0$ (respectively, $T_1$), then we have found a trade. 
We add it to the list of found solutions and continue the search.

We finish this section by the pseudocode of the algorithm.

\mbox{}\\
\verb|define RECURSION(|$s$\verb|): #| \emph{$s$ is the step number} \\
\verb|       |$j:=s\bmod 2$ \verb| #| \emph{the parity of the step} \\
\verb|       |$i:=1-j$ \\
\verb|       if |$T_0^{+}=T_1^{+}=\{\}$\verb|: |  \\
\verb|             FOUND_SOLUTION() #| \emph{record the solution $(T_0,T_1)$, proceed isomorph rejection, \ldots} \\
\verb|       else if |$T_i^{+}=\{\}$\verb|:| \\
\verb|             RECURSION(|$s+1$\verb|) #| \emph{go to the next step} \\
\verb|       else:  | \\
\verb|             choose |$v$\verb| from |$T_i^+$ \\
\verb|             |$T_i^+:=T_i^+\backslash \{v\}$ \\ 
\verb|             for all 5-subsets |$N$\verb| of the neighborhood of |$v$ \\ 
\verb|              such that |$N\cup T_{j}$\verb| is an independent set do:| \\ 
\verb|                     |$N^{+}:=N\backslash T_{j}$\verb| #| \emph{new vertices to add} \\ 
\verb|                     |$T_{j}^{+}:=T_{j}^{+} \cup N^+$ \\ 
\verb|                     |$T_{j}:=T_{j} \cup N^+$\\ 
\verb|                     |$T_{i}:=T_{i} \cup (N^+ + 1111111111)$ \\ 
\verb|                     RECURSION(|$s$\verb|) | \\ 
\verb|                     |$T_{j}^{+}:=T_{j}^{+} \backslash N^+$ \\ 
\verb|                     |$T_{j}:=T_{j} \backslash N^+$ \\ 
\verb|                     |$T_{i}:=T_{i} \backslash (N^+ + 1111111111)$ \\ 
\verb|             |$T_i^+:=\{v\} \cup T_i^+ $ \\ 
\verb|#| \emph{now, the main part of the algorithm}
\\
$T_0 := \{0000000000, 1111111100,1111110011,1111001111,1100111111,0011111111\}$
\\
$T_1 := \{1111111111, 0000000011,0000001100,0000110000,0011000000,1100000000\}$
\\
$T_0^{+} : = \{ \}$ \verb|#|\ \  \emph{$T_0^{+}$ and $T_1^{+}$ keep the chosen vertices with the ``unsolved'' neighborhood}
\\
$T_1^{+} := \{0000000011,0000001100,0000110000,0011000000,1100000000\}$
\\
\verb|RECURSION(2)|

\subsection{Results}\label{ss:res}
There are $8$ nonequivalent extended $1$-perfect trades in \HammingGraphalvedCube{10}.
Below, for each of them, we list the component $T_0$, 
while $T_1$ is obtained by taking the complement for each word of $T_0$.
For briefness, $T_0$ is represented in the form $T_0=K+R=\{a+b\,:\,a\in K,b\in R\}$, 
where $K$ (the \emph{kernel} of $T_0$) 
is the maximal linear subspace admitting such decomposition of $T_0$.

T16:
%$0000000000$,
%$0000001111$,
%$0000110011$,
%$0000111100$,
%$0011000011$,
%$0011001100$,
%$0011110000$,
%$0011111111$,
%$1100000011$,
%$1100001100$,
%$1100110000$,
%$1100111111$,
%$1111000000$,
%$1111001111$,
%$1111110011$,
%$1111111100$.
%
$\langle 0000001111,0000110011,0011000011,1100000011 \rangle+0101010101$.

T24:
$\langle 1111\,0000\,00,0000\,1111\,00 \rangle+\{
 0011\,0011\,01$, $0101\,0101\,01$, $0110\,0110\,01$,
$0011\,0101\,10$, $0101\,0110\,10$, $0110\,0011\,10\}$.

% 0000 000000$,
%  $0000 001111$,
%$0000 110011$,
%$0000 111100$,
%  $0011 000011$,
%  $0011 010100$,
%$0011 101000$,
%$0011 111111$,
%  $0101 001100$,
%  $0101 010111\}$.
%$0101 101011$,
%$0101 110000$,
%$1010 001100$,
%$1010 010111$,
%$1010 101011$,
%$1010 110000$,
%$1100 000011$,
%$1100 010100$,
%$1100 101000$,
%$1100 111111$,
%$1111 000000$,
%$1111 001111$,
%$1111 110011$,
%$1111 111100$.

T28:
$\langle 11111111\,00 \rangle+\{
 00010111\,01$,
$00101110\,01$,
$01011100\,01$,
$00111001\,01$,
$01110010\,01$,
$01100101\,01$,
$01001011\,01$,
$00011101\,10$,
$00111010\,10$,
$01110100\,10$,
$01101001\,10$,
$01010011\,10$,
$00100111\,10$,
$01001110\,10\}$,
%$0000000000$,
%$0000001111$,
%$0000110011$,
%$0001011100$,
%$0010110100$,
%$0011000011$,
%$0011101000$,
%$0011111111$,
%$0100101100$,
%$0101010111$,
%$0101101011$,
%$0101110000$,
%$0110100111$,
%$0111000100\}$.
%$1000111000$,
%$1001011011$,
%$1010001100$,
%$1010010111$,
%$1010101011$,
%$1011010000$,
%$1100000011$,
%$1100010100$,
%$1100111111$,
%$1101001000$,
%$1110100000$,
%$1111001111$,
%$1111110011$,
%$1111111100$.

T32a:
$\langle 1111000000,0000111100,0110011000 \rangle +
\{0000 000000$,
$0000 001111$,
%$0000 110011$,
%$0000 111100$,
$0011 000011$,
$0011 010100\}$.
%$0011 101000$,
%$0011 111111$,
%$0101 001100$,
%$0101 011011$,
%$0101 100111$,
%$0101 110000$,
%$0110 010111$,
%$0110 011000$,
%$0110 100100$,
%$0110 101011$,
%$1001 010111$,
%$1001 011000$,
%$1001 100100$,
%$1001 101011$,
%$1010 001100$,
%$1010 011011$,
%$1010 100111$,
%$1010 110000$,
%$1100 000011$,
%$1100 010100$,
%$1100 101000$,
%$1100 111111$,
%$1111 000000$,
%$1111 001111$,
%$1111 110011$,
%$1111 111100$.

T32b:
$\langle 1111111100,0110011000 \rangle +
\{0000000000$,
$0000001111$,
$0000110011$,
$0001011100$,
$0010110100$,
$0011000011$,
$0011101000$,
$0011111111\}$.
%$0100101100$,
%$0101011011$,
%$0101100111$,
%$0101110000$,
%$0110010111$,
%$0110011000$,
%$0110101011$,
%$0111000100$,
%$1000111000$,
%$1001010111$,
%$1001100100$,
%$1001101011$,
%$1010001100$,
%$1010011011$,
%$1010100111$,
%$1011010000$,
%$1100000011$,
%$1100010100$,
%$1100111111$,
%$1101001000$,
%$1110100000$,
%$1111001111$,
%$1111110011$,
%$1111111100$.

T32c:
$\langle 0000111100 \rangle+
\{0000 000000$,
  $0000 001111$,
% $0000 110011$,
% $0000 111100$,
  $0001 011001$,
% $0001 100101$,
  $0011 000011$,
  $0011 010100$,
% $0011 101000$,
% $0011 111111$,
  $0101 001100$,
  $0101 010111$,
% $0101 101011$,
% $0101 110000$,
  $0111 011010$,
% $0111 100110$,
  $1000 010101$,
% $1000 101001$,
  $1010 001100$,
  $1010 011011$,
% $1010 100111$,
% $1010 110000$,
  $1100 000011$,
  $1100 011000$,
% $1100 100100$,
% $1100 111111$,
  $1110 010110$,
% $1110 101010$,
  $1111 000000$,
  $1111 001111\}$.
% $1111 110011$,
% $1111 111100$.

T36:
$\{0\,111\,100\,001$,
$  0\,111\,001\,010$,
$  0\,111\,010\,100$,
$  0\,100\,001\,111$,
$  0\,001\,010\,111$,
$  0\,010\,100\,111$,
$  0\,001\,111\,100$,
$  0\,010\,111\,001$,
$  0\,100\,111\,010$,
$  0\,101\,100\,110$,
$  0\,011\,001\,101$,
$  0\,110\,010\,011$,
$  0\,100\,110\,101$,
$  0\,001\,101\,011$,
$  0\,010\,011\,110$,
$  0\,110\,101\,100$,
$  0\,101\,011\,001$,
$  0\,011\,110\,010$,
$  1\,000\,110\,011$,
$  1\,000\,101\,110$,
$  1\,000\,011\,101$,
$  1\,011\,000\,110$,
$  1\,110\,000\,101$,
$  1\,101\,000\,011$,
$  1\,110\,011\,000$,
$  1\,101\,110\,000$,
$  1\,011\,101\,000$,
$  1\,010\,001\,011$,
$  1\,100\,010\,110$,
$  1\,001\,100\,101$,
$  1\,011\,010\,001$,
$  1\,110\,100\,010$,
$  1\,101\,001\,100$,
$  1\,001\,011\,010$,
$  1\,010\,110\,100$,
$  1\,100\,101\,001\}$.
%$\{0000000000$,
%$0000010111$,
%$0000101011$,
%$0001001101$,
%$0001110100$,
%$0010001110$,
%$0010111000$,
%$0011010001$,
%$0011100010$,
%$0011111111$,
%$0100011100$,
%$0100110010$,
%$0101000011$,
%$0101101110$,
%$0101111001$,
%$0110011011$,
%$0111001000$,
%$0111010110$,
%$1000101100$,
%$1000110001$,
%$1001100111$,
%$1010000011$,
%$1010011101$,
%$1010110110$,
%$1011000100$,
%$1011101001$,
%$1100000110$,
%$1100001001$,
%$1100111111$,
%$1101010101$,
%$1101100000$,
%$1110010000$,
%$1110101010$,
%$1111001111$,
%$1111110011$,
%$1111111100\}$.

T40:
$\{0000000000$,
$0000001111$,
$0000110101$,
$0001010011$,
$0001011100$,
$0001101010$,
$0010011010$,
$0010100110$,
$0010101001$,
$0011000101$,
$0011110000$,
$0011111111$,
$0100100011$,
$0100111000$,
$0101101101$,
$0101110110$,
$0110001100$,
$0110010111$,
$0111000010$,
$0111011001$,
$1000101100$,
$1000110010$,
$1001100111$,
$1001111001$,
$1010000011$,
$1010011101$,
$1011001000$,
$1011010110$,
$1100000110$,
$1100001001$,
$1100111111$,
$1101010101$,
$1101011010$,
$1101100000$,
$1110010000$,
$1110100101$,
$1110101010$,
$1111001111$,
$1111110011$,
$1111111100\}$.

The trades T16, T24, T28, and T36 are constant-weight; 
the others, T32a, T32b, T32c, and T40, cannot be represented as constant-weight. 
Each mate of the trade T40 is an optimal distance-$4$ code 
equivalent to the Best code \cite{Best80}.

Table~\ref{t:1} reflects some properties of the listed trades.
\begin{table}
\centering
\begin{tabular}{@{}l|c|c|l|l}
Name& Volume & Rank & $|\mathrm{Aut}|$ &Coordinate orbits
\\\hline\hline
T16 &16& 4+1 & $2\cdot 16\cdot 3840$ &  \{0,1,2,3,4,5,6,7,8,9\}
\\\hline
T24 &24& 7+0 & $2\cdot 24\cdot 64$ & \{0,1\}, \{2,3,4,5,6,7,8,9\}
\\\hline
T28 &28& 8+0 & $2\cdot 28\cdot 48$ & \{0,1\}, \{2,3,4,5,6,7,8,9\}
\\\hline
T32a &32 &6+1 & $2\cdot 32\cdot 64$ & \{0,1\}, \{2,3,4,5,6,7,8,9\}
\\\hline
T32b &32& 7+1 & $2\cdot 32\cdot 48$ & \{0,1,6,9\}, \{2,3,4,5\}, \{7,8\}
\\\hline
T32c &32& 8+0 & $2\cdot 16\cdot 8$ & \{0,2,4,5\}, \{1,3\}, \{6,7,8,9\}
\\\hline
T36 &36& 9+0 & $2\cdot 36\cdot 40$ & \{0,1,2,3,4,5,6,7,8,9\}
\\\hline
T40 &40& 9+0 & $2\cdot 40\cdot 8$ & \{0,1,2,3,4,5,6,7,8,9\}
\end{tabular}
\caption{Extended $1$-perfect trades in the $10$-cube}
\label{t:1}
\end{table}
For each trade $(T_0,T_1)$, 
the second column contains its volume.
The column ``rank'' contains the affine rank of $T_0\cup T_1$,
where the first summand is the rank of $T_0$ 
(as follows from Corollary~\ref{c:antipod}, the second summand is either $0$ or $1$ for length $10\equiv 2\bmod 4$).

The column $|\mathrm{Aut}|$ contains the order of the automorphism group of $T_0\cup T_1$.
In this column, the last factor is the order of the 
stabilizer of a vertex of $T_0\cup T_1$ under  $\mathrm{Aut}(T_0\cup T_1)$;
and the last two factors correspond to 
$|\mathrm{Aut}(T_0)|$ (it can be seen from Corollary~\ref{c:antipod} that $|\mathrm{Aut}(T_0\cup T_1)|=2\cdot|\mathrm{Aut}(T_0)|$).
In seven (all but one) cases, the second factor coincides with $|T_0|$. 
For these seven trades, $T_0$ forms one orbit under the action $\mathrm{Aut}(T_0)$; i.e., $\mathrm{Aut}(T_0)$ acts transitively on $T_0$.
For one trade, T32c, $T_0$ is divided into two orbits of size $16$.

The last column contains the %sizes of 
coordinate orbits under the action of $\mathrm{Aut}(T_0\cup T_1)$. 
In particular, the number of orbits is the number of nonequivalent $1$-perfect trades in \HammingGraph{9} obtained
by puncturing (deleting the same coordinate in all words) from a given extended $1$-perfect trade in \HammingGraphalvedCube{10}.
We see that the total number of nonequivalent primary $1$-perfect trades in \HammingGraph{9} is $15 = 1+2+2+2+3+3+1+1$.

In Table~\ref{t:2}, we list all STS trades derived from extended $1$-perfect trades of length $10$.
In  \cite[Table~3.4]{FGG:2004:trades}, the authors list all nonequivalent STS trades of volume at most $9$.
As the result of the current search, we can say that the STS trades number 
1 (of volume $4$), 2, 4 (of volume $6$), 5 (a pair of STS of volume $7$), 7, 11--16 (of volume $8$) are derived,
6 (of volume $7$) and 10 (of volume $8$) are not derived 
(numbers 3, 8, 9 in \cite[Table~3.4]{FGG:2004:trades} are for $3$- and $4$-way trades),
all STS trades of volume $9$ are not derived from extended $1$-perfect trades of length $10$.
The number in a cell of the table indicates how many times the STS trade corresponding to the row
occurs in the extended $1$-perfect trade $(T_0,T_1)$ corresponding to the column 
(i.e. the number of $x$ such that the weight-$3$ words of $T_0+x$, $T_1+x$ form the corresponding STS trade).
Note that derived STS trades are not necessarily primary.
\begin{table}
\centering
\begin{tabular}{@{\,}r|l|@{\,}c@{\,}|@{\,}c@{\,}|@{\,}c@{\,}|@{\,}c@{\,}|@{\,}c@{\,}|@{\,}c@{\,}|@{\,}c@{\,}|@{\,}c@{\,}|@{\,}}
No &blocks & \rotatebox{90}{T16} & \rotatebox{90}{T24} & \rotatebox{90}{T28} & \rotatebox{90}{T32a} & \rotatebox{90}{T32b} & \rotatebox{90}{T32c} & 
\rotatebox{90}{T36} & \rotatebox{90}{T40}
\\ \hline 1 &
012, 034, 135, 245 &&&&&&&&\\ & 013, 024, 125, 345
&320& 48& 0& 0& 0& 16&  0& 0\\ \hline 2 &
012, 034, 135, 146, 236, 245 &&&&&&&&\\ &
013, 024, 126, 145, 235, 346
&0& 64& 0& 0& 0& 0&  0& 0\\ \hline 4 &
012, 034, 135, 246, 257, 367 &&&&&&&&\\ &
013, 024, 125, 267, 346, 357
&0& 128& 112& 0& 0& 0&  0& 0\\ \hline 5 &
012, 034, 056, 135, 146, 236, 245 &&&&&&&&\\ &
013, 025, 046, 126, 145, 234, 356
&0& 0& 32& 0& 0& 0&  0& 0\\ \hline 7 &
%012, 034, 067, 135, 168, 245, 378 &&&&&&&&\\ &
%016, 024, 037, 125, 138, 345, 678
%&0& 0& 0& 0& 0& 0&  0& 0\\ \hline &
012, 034, 067, 135, 147, 236, 257, 456 &&&&&&&&\\ &
013, 026, 047, 127, 145, 235, 346, 567
&0& 0& 0& 0& 0& 0&  90& 0\\ \hline 11 &
%012, 034, 057, 135, 146, 236, 278, 568 &&&&&&&&\\ &
%014, 027, 035, 123, 156, 268, 346, 578
%&0& 0& 0& 0& 0& 0&  0& 0\\ \hline &
012, 034, 135, 146, 178, 236, 247, 258 &&&&&&&&\\ &
013, 024, 126, 147, 158, 235, 278, 346
&0& 0& 0& 0& 0& 64&  0& 0\\ \hline 12 &
012, 034, 135, 147, 236, 258, 378, 468 &&&&&&&&\\ &
014, 023, 125, 137, 268, 346, 358, 478
&0& 0& 0& 0& 0& 32&  0& 0\\ \hline 13 &
012, 034, 135, 146, 178, 236, 379, 589 &&&&&&&&\\ &
014, 023, 126, 137, 158, 346, 359, 789
&0& 0& 0& 0& 0& 16&  0& 0\\ \hline 14 &
012, 034, 067, 089, 135, 245, 568, 579 &&&&&&&&\\ &
013, 024, 068, 079, 125, 345, 567, 589
&0& 24& 0& 64& 0& 0&  0& 0\\ \hline 15 &
012, 034, 135, 246, 257, 289, 368, 379 &&&&&&&&\\ &
013, 024, 125, 268, 279, 346, 357, 389
&0& 0& 84& 128& 192& 0&  0& 0\\ \hline 16 &
012, 034, 135, 246, 257, 368, 589, 679 &&&&&&&&\\ &
013, 024, 125, 267, 346, 358, 579, 689
&0& 0& 0& 0& 0& 32&  0& 0\\ \hline &
023, 045, 124, 135, 258, 348, 068, 079, 169, 178 &&&&&&&&\\ &
025, 034, 123, 145, 248, 358, 069, 078, 168, 179
&0& 0& 0& 0& 0& 32&  0& 0\\ \hline &
035, 079, 048, 127, 145, 168, 269, 258, 349, 367, &&&&&&&&\\ &
037, 058, 049, 125, 148, 167, 279, 268, 369, 345, 
&0& 0& 0& 0& 0& 0&  0& 32\\ \hline &
017, 029, 038, 128, 145, 139, 235, 367, 468, 479, 578, 569 &&&&&&&&\\ &
019, 037, 028, 125, 147, 138, 239, 356, 458, 579, 678, 469
&0& 0& 0& 0& 0& 0&  0& 80\\ \hline &
017, 028, 039, 056, 129, 145, 168, 235, 247, 367, 469, 348  &&&&&&&&\\ &
018, 029, 035, 067, 127, 149, 156, 238, 245, 347, 369, 468,            
&0& 0& 0& 0& 0& 0&  60& 0\\ \hline
\end{tabular}
\caption{STS trades derived from extended $1$-perfect trades of length $10$;
the numbers in the first column are given in accordance with \cite[Table~3.4]{FGG:2004:trades}.}
\label{t:2}
\end{table}

It should be noted that if an STS trade is derived from extended $1$-perfect trades of length $n+1$,
and has at least one constantly zero coordinate, then it is derived from $1$-perfect trades of length $n$ (in our case, $n=9$). This argument is applicable to the first seven STS trades in Table~\ref{t:2}.

\subsection{Validation of classification}\label{ss:valid}
To check the results, we recount the number of solutions that should be found by the algorithm in an alternative way.
Double-counting is a standard way to validate computer-aided classifications of combinatorial objects, see \cite{KO:alg}.

Given a trade $(T_0,T_1)$, consider all graph automorphisms that send it to a solution.
For every word $t$ from $T_0\cup T_1$ and its five neighbors $t_1$, $t_2$, $t_3$, $t_4$, $t_5$ from $T_0\cup T_1$,
there is one translation $x \to x+t$ that sends $t$ to $0^{10}$ and $5!\cdot 2^5$ coordinate permutations that send 
$\{t_1, t_2, t_3, t_4, t_5\}$ to $\{v_1, v_2, v_3, v_4, v_5\}$. 
So, totally, there are $|T_0\cup T_1|\cdot 5!\cdot 2^5$ graph automorphisms that make from $(T_0,T_1)$ one of the solutions of the algorithm above.
Then, the number of different solutions equivalent to $(T_0,T_1)$ is
$$ |T_0\cup T_1|\cdot 5!\cdot 2^5 / |\mathrm{Aut}(T_0\cup T_1)|.$$
Summing this value over all found nonequivalent trades, 
we get $1817$, the exact number of different solutions found by the computer.

\section{Extended $1$-perfect trades in \HammingGraphalvedCube{12}}\label{s:12}
It is hardly possible to enumerate all primary extended $1$-perfect trades in \HammingGraphalvedCube{12} using the technique described above,
even if we reject isomorphic partial solutions at some steps of the search.
However, if we restrict the search by only the words of weight $6$, 
the number of cases becomes essentially smaller 
and exhaustive enumeration becomes possible
if we additionally apply isomorph rejection.
The idea of this technique is standard: at some stage,
we check the obtained partial solution and reject it 
if it is equivalent to a partial solution considered before.
Similarly to the length-$10$ case, we fix one element of $T_0$,
now it is $000000111111$, and six its $T_1$-neighbors, $v_1$, \ldots, $v_6$.
After some experiments, it was decided to perform isomorph rejection
after choosing the $T_0$-neighbors for $v_1$, $v_2$, $v_3$ and 
after choosing the $T_0$-neighbors for all $v_1$, \ldots, $v_6$.
The isomorph rejection reduced the total time of the algorithm run by the factor $1400$, approximately.
All calculation took an hour and a half % several days 
using one core of a 3GHz personal computer.
The classification was validated using the same approach as for the length $10$ (Section~\ref{ss:valid}); 
however, taking into account the isomorph rejection, each solution found by the computer was counted with the multiplicity 
equal the multiplicity of the corresponding partial solution. 
The total number of solutions, taking into account the multiplicities, is $32076$, 
and it coincides with the expected number
calculated from the orders of the automorphism groups of the nonequivalent solutions.

\subsection{Description of the trades}\label{ss:descr}
The results of the classification are the following.
Up to equivalence, there are exactly $25$ constant-weight extended $1$-perfect trades in \HammingGraphalvedCube{12} of the following
volumes:
$32$, 
$48$, 
$56$, $56$, 
$68$, 
$86$, 
$72$, $72$, $72$, $72$, 
$80$, $80$, 
$92$, $92$, 
$92$, $96$, 
$96$, $98$, 
$102$, 
$108$, $108$, 
$110$, $110$, 
$120$, $120$, 
$132$.
The data for generation of the trades can be found in the table below.

The first column contains the volume of the trade, sometimes followed by a
letter, to form a unique ``name''.

Representatives of the orbits of $T_0$ are in the column ``$T_0$'' of the table.
The number in the index indicates the size of the orbit; 
sole number means that the orbit is self-complementary 
(i.e., each element is contained together with its complement); 
if the index ends with $\cdot 2$, then the complementary orbit should be 
additionally taken. 

The column marked $\mathrm{Sym}(T_0)\cap\mathrm{Sym}(T_1)$ contains generators of 
this group, and, sometimes an information about its structure. 
The grayed generators are not necessary to induce $T_0$
(i.e., all orbits are induced 
by only the subgroup generated by non-grayed elements).
In each case, the coordinates are ordered 
in such a way that the symmetry group can be represented in a convenient intuitive way
(as much as possible, from the point of view of the author).

If $\mathrm{Sym}(T_0 \cup T_1)=\mathrm{Sym}(T_0)\cap\mathrm{Sym}(T_1)$, 
then representatives of the orbits of $T_1$ are listed in the 
``$T_1$'' column. Usually this means that $T_0$ and $T_1$ are nonequivalent;
the only exception is the trade 80a, where $T_1=T_0+000000111111$.
The other case is
$|\mathrm{Sym}(T_0 \cup T_1)| = 2|\mathrm{Sym}(T_0)\cap\mathrm{Sym}(T_1)|$; then the column
``$T_1$'' contains an additional generator element.
The same column contains information how to generate $T_2$ in the case when
the trade $(T_0,T_1)$ can be expanded to a $3$-way trade $(T_0,T_1,T_2)$.

The last column of the table contains: (1) the order of the automorphism group of the trade, represented in the form 
$|\{x \mid x+T_0\cup T_1 = T_0\cup T_1\}|\cdot |\mathrm{Sym}(T_0\cup T_1)|$, 
where $|\mathrm{Sym}(T_0\cup T_1)$ is the stabilizer of $0^{12}$ in $\mathrm{Aut}(T_0\cup T_1)$
(for all considered trades, the automorphism group happens to be the product of the translation group with $\mathrm{Sym}$);
(2) the standard basis of the dual space; 
(3) the mark ``Witt'' 
if the $T_0$ is a subset of S$(5,6,12)$ (see the next subsection);
(4) the mark ``no squares'' for the unique trade whose graph has girth more than $4$, i.e., $6$.

\newcommand{\mlns}[2][l]{%
  $\begin{array}{@{}#1@{}}#2\end{array}$}
\newcommand\rbx[1]{\raisebox{-1em}{\rotatebox{90}{#1}}}

\begin{longtable}{l|c|c|c|l}
  & $T_0$ & $\mathrm{Sym}(T_0)\cap\mathrm{Sym}(T_1)$ & $T_1$ & properties
 \endhead
 \hline\hline
%------------------------------------------
 \rbx{32} &
 \mlns{01\,01\,01\,01\,01\,01_{\times 32}} &
 \mlns[c]{(89)(ab) \quad (8a)(9b)\\(02468)(13579)} &
 \mlns{(ab)} &
 \mlns{|\mathrm{Aut}|=64\cdot 46080 \\
       \mbox{dual: } 0^{2i}1^20^{10-2i},\\ \ \  i=0,1,2,3,4,5 }
 \\ \hline
%------------------------------------------
 \rbx{48} &
% \mlns{01\,01\,0011\,0101_{\times 48}} &
 \mlns{0101\,0011\,01\,01_{\times 48}} &
 \mlns[c]{(012)(465)\\(01)(23)\quad 
          (45)(67)\\(02)(46)(89)\\
          (89)(ab)\quad \color[rgb]{0.4,0.4,0.4}(8a)(9b)\\ 
          \color[rgb]{0.4,0.4,0.4}(04)(15)(26)(37)} &
 \mlns{(ab)} &
 \mlns{|\mathrm{Aut}|=16\cdot 1536 \\
       \mbox{dual: } 1^40^8,\ 0^41^40^4,\\\ \ 0^81^20^2,\ 0^{10}1^2}
 \\ \hline
%------------------------------------------
 \rbx{56a} &
% \mlns{010100011101_{\times 56}} &
 \mlns{01000111\,01\,01_{\times 56}} &
 \mlns[c]{(0123456)\\
       (13)(24)(67)(89)\\(89)(ab) \qquad  
       \color[rgb]{0.4,0.4,0.4} (8a)(9b)} &
 \mlns{(ab)} &
 \mlns{|\mathrm{Aut}|=8 \cdot 2688 \\ 
 \mbox{dual: }\\ \ \  1^80^4,\ 0^81^20^2,\ 0^{10}1^2}
 \\ \hline
%------------------------------------------
 \rbx{56b} &
 \mlns{0110\,0110\,0110_{\times 8 }\\
       0011\,0101\,0101_{\times 24}\\
       0011\,0011\,0110_{\times 24}\\} &
 \mlns[c]{(01)(23)\qquad(02)(13)\\(048)(159)(26a)(37b)\\ \color[rgb]{0.4,0.4,0.4} (04)(15)(26)(37)} &
 \mlns{(01)(45)(89)} &
 \mlns{|\mathrm{Aut}|=8 \cdot 768 \\
       \mbox{dual: }\\ \ \  1^40^8,\ 0^41^40^4,\ 0^81^4}
 \\  \hline
%------------------------------------------
 \rbx{68} &
%  \mlns{0110\,0000\,1111_{\times 4 }\\   
%        0101\,0011\,0011_{\times 12}\\   
%        0101\,0011\,1100_{\times 12}\\   
%        0011\,0001\,1110_{\times 16}\\   
%       0110\,0011\,0110_{\times 24}} &
 \mlns{0000\,1111\,0110_{\times 4 }\\   
       0011\,0011\,0101_{\times 12}\\   
       0011\,1100\,0101_{\times 12}\\   
       0001\,1110\,0011_{\times 16}\\   
       0011\,0110\,0110_{\times 24}} &
 \mlns[c]{(123)(567)\\(05)(14)(26)(37)\\(89)(ab)\quad(8a)(9b)\\ \sim S_4\times C_2^2}&
%  \mlns{0011\,0000\,1111_{\times 4 }\\   
%        0110\,0011\,0011_{\times 12}\\   
%        0011\,0011\,1100_{\times 12}\\   
%        0110\,0001\,1110_{\times 16}\\   
%        0101\,0011\,0110_{\times 24}} &
 \mlns{0000\,1111\,0011_{\times 4 }\\   
       0011\,0011\,0110_{\times 12}\\   
       0011\,1100\,0011_{\times 12}\\   
       0001\,1110\,0110_{\times 16}\\   
       0011\,0110\,0101_{\times 24}} &
 \mlns{|\mathrm{Aut}|=4 \cdot 96 \\
       \mbox{dual: }1^80^4,\ 0^41^8}
 \\ \hline
%------------------------------------------
 \rbx{72a} &
 \mlns{0011\,0011\,0011_{\times 72}} &
 \mlns[c]{(01)(23)\quad (012)(465)\\(048)(159)(26a)(37b)
       \\ \color[rgb]{0.4,0.4,0.4}(04)(15)(26)(37)} &
 \mlns{(01)(45)(8a)\\ T_2: (01)(45)(9a)} &
 \mlns{|\mathrm{Aut}|=8 \cdot 6912\\
       \mbox{dual: }\\ \ \ 1^40^8,\ 0^41^40^4,\ 0^81^4}
 \\ \hline
 %------------------------------------------
 \rbx{72b} &
 \mlns{0000111101\,01_{\times 72}} &
 \mlns[c]{(0187)(2365)\\(09)(48)(57)(ab)} &
 \mlns{(ab)} &
 \mlns{|\mathrm{Aut}|=4 \cdot 2880\\
       \mbox{dual: }1^{10}0^2,\ 0^{10}1^2\\
       \mbox{Witt} }
 \\ \hline
 %------------------------------------------
 \rbx{72c} &
 \mlns{0110\,1001\,0011_{\times 12}\\
       0101\,0101\,0011_{\times 12}\\
       0001\,1011\,0011_{\times 48}} &
 \mlns[c]{(01)(23)(45)(67)\\(89)(ab)\\(123)(567)(9ab)\\(04)(15)(26)(37)} &
 \mlns{(01)(45)(89)} &
 \mlns{|\mathrm{Aut}|=4 \cdot 192\\
 \mbox{dual: }1^80^4,\ 0^81^4}
 \\[7mm] \hline
%------------------------------------------
 \rbx{72d} &
 \mlns{001\,101\,010\,011_{\times 36}\\
       001\,101\,011\,100_{\times 36}} &
 \mlns[c]{(012)(678)\\(12)(45)(78)(ab)\\(03)(14)(25)(69)(7a)(8b)\\\color[rgb]{0.4,0.4,0.4}(06)(17)(28)(39)(4a)(5b)} &
 \mlns{(12)(78)} &
 \mlns{|\mathrm{Aut}|=4 \cdot 144\\
 \mbox{dual: }1^60^6,\ 0^61^6}
 \\[7mm] \hline
%------------------------------------------
 \rbx{80a} &
 \mlns{000111\,111000_{\times 20}\\
       000111\,010110_{\times 60}} &
 \mlns[c]{(01234)(6789a)\\(1342)(79a8)\\
       (05)(23)(6b)(89)\\\color[rgb]{0.4,0.4,0.4}(06)(17)(28)(39)(4a)(5b)
       \\\sim PGL_2(5)\times C_2} &
 \mlns{000111\,000111_{\times 20}\\
       000111\,101001_{\times 60}\\
       T_1=T_0+0^6 1^6} &
 \mlns{|\mathrm{Aut}|=4 \cdot 240 \\
 \mbox{dual: }1^60^6,\ 0^61^6}
 \\[7mm] \hline
%------------------------------------------
  \rbx{80b} &
 \mlns[l]{0000\,1111\,0110_{\times 4 }\\
          0011\,0011\,0011_{\times 12}\\
          0001\,1110\,0011_{\times 16}\\
          0011\,0110\,0101_{\times 24}\\
          0011\,0101\,0110_{\times 24}} &
 \mlns[c]{(01)(23)(45)(67)\\(123)(567)\\(04)(15)(27)(36)\\(89)(ab)\quad (8a)(9b)\\\sim S_4\times C_2^2} &
 \mlns{(23)(67)(ab)} &
 \mlns{|\mathrm{Aut}|=4 \cdot 192\\
 \mbox{dual: }1^80^4,\ 0^81^4}
 \\[7mm] \hline
%------------------------------------------
  \rbx{86} &
 \mlns{000\,111\,111\,000_{\times 2 }\\
       000\,111\,001\,110_{\times 12}\\
       001\,100\,011\,101_{\times 18}\\
       001\,001\,011\,011_{\times 18}\\
       001\,011\,011\,100_{\times 36} } &
 \mlns[c]{(012)(345)\\(05)(14)(23)(6b)(7a)(89)\\(06)(17)(28)(39)(4a)(5b)} &
 \mlns{000\,111\,000\,111_{\times 2 }\\
       000\,111\,011\,100_{\times 12}\\
       001\,011\,011\,010_{\times 18}\\
       001\,110\,001\,110_{\times 18}\\
       001\,001\,011\,101_{\times 18\cdot 2} } &
\mlns{|\mathrm{Aut}|=2 \cdot 36\\ \mbox{dual: }1^{12}}
 \\[7mm] \hline
%------------------------------------------
 \rbx{92a} &
 \mlns{000000\,111111_{\times 2 }\\
        001111\,110000_{\times 30}\\
        000111\,100011_{\times 60} } &
 \mlns[c]{(01234)(6789a)\\(1342)(79a8)\\(05)(23)(6b)(89)\\(06)(17)(28)(39)(4a)(5b)\\\sim PGL_2(5)\times C_2} &
 \mlns{ 011111\,100000_{\times 12}\\
        000111\,000111_{\times 20}\\
        000111\,101001_{\times 60} } &
\mlns{|\mathrm{Aut}|=2 \cdot 240\\ 
      \mbox{dual: }1^{12}}
 \\[7mm] \hline
%------------------------------------------
 \rbx{92b} &
 \mlns[l]{00\,00\,00\,\,11\,11\,11_{\times 2 }\\
           00\,00\,11\,\,11\,11\,00_{\times 6 }\\
           00\,01\,11\,\,00\,01\,11_{\times 12}\\
           00\,01\,01\,\,11\,01\,01_{\times 24}\\
           00\,01\,11\,\,10\,10\,01_{\times 48} } &
 \mlns[c]{(0213)(6879)\\(2435)(8a9b)\\(06)(17)(28)(39)(4a)(5b)\\ \sim S_4\times C_2} &
 \mlns{
01\,01\,01\,\,10\,10\,10_{\times 8 }\\
00\,01\,11\,\,11\,10\,00_{\times 12}\\
00\,01\,11\,\,00\,10\,11_{\times 12}\\
00\,00\,01\,\,11\,11\,01_{\times 12}\\
00\,01\,11\,\,10\,01\,01_{\times 48} } &
\mlns{|\mathrm{Aut}|=2 \cdot 48\\ 
      \mbox{dual: }1^{12}}
 \\[7mm] \hline
%------------------------------------------
 \rbx{96a} &
 \mlns[l]{ 0001\,0011\,0111} &
 \mlns[c]{(0527)(1634)\\(048)(159)(26a)(37b)\\ \color[rgb]{0.4,0.4,0.4}(01)(23)(45)(67)(89)(ab)} &
 \mlns{(02)(13)} & 
\mlns{|\mathrm{Aut}|=2 \cdot 384\\ \mbox{dual: }1^{12} \\
      \mbox{Witt}}
 \\[7mm] \hline
%------------------------------------------
 \rbx{96b} &
 \mlns[l]{00\,01\,11\,\,00\,10\,11_{\times 12}\\ 
           00\,00\,01\,\,11\,11\,01_{\times 12}\\ 
           00\,01\,01\,\,11\,10\,10_{\times 24}\\ 
           00\,01\,11\,\,01\,01\,10_{\times 48} } &
 \mlns[c]{(0213)(6879)\\(2435)(8a9b)\\(06)(17)(28)(39)(4a)(5b) \\ \sim S_4\times C_2 } &
 \mlns{00\,01\,11\,\,00\,01\,11_{\times 12}\\ 
       00\,00\,01\,\,11\,11\,10_{\times 12}\\ 
       00\,01\,01\,\,11\,01\,01_{\times 24}\\ 
       00\,01\,11\,\,01\,10\,10_{\times 48} } &
\mlns{|\mathrm{Aut}|=2 \cdot 48\\ 
      \mbox{dual: }1^{12}}
 \\[7mm] \hline
%------------------------------------------
 \rbx{98} &
 \mlns[l]{111111\,000000_{\times 2   }\\
           001110\,100011_{\times 6   }\\
           011011\,001001_{\times 6   }\\
           010110\,001101_{\times 6\cdot 2 }\\
           011010\,010110_{\times 6\cdot 2 }\\
           001100\,111001_{\times 12  }\\
           010100\,110101_{\times 12  }\\
           001010\,110101_{\times 12  }\\
           000101\,100111_{\times 12\cdot 2}\\ } &
 \mlns[c]{(012345)(6789ab)\\(0b)(1a)(29)(38)(47)(56)\\ \sim D_6} &
 \mlns[l]{010101\,101010_{\times 2   }\\
          011100\,110001_{\times 6   }\\
          011100\,001110_{\times 6   }\\
          001011\,001011_{\times 12  }\\
          010110\,001011_{\times 12  }\\
          001011\,010110_{\times 12  }\\
          000001\,101111_{\times 12  }\\
          000101\,010111_{\times 12  }\\
          000101\,111001_{\times 12\cdot 2} } &
\mlns{|\mathrm{Aut}|=2 \cdot 12\\ 
       \mbox{dual: }1^{12}}
 \\[7mm] \hline
%------------------------------------------
 \rbx{102} &
 \mlns{10\,01\,01\,01\,10\,01_{\times 12}\\
        10\,01\,11\,00\,01\,10_{\times 30}\\
        11\,11\,00\,10\,00\,10_{\times 30\cdot 2} } &
 \mlns[c]{(01234)(6789a)\\(01)(67)(35)(9b)(28)(4a)\\ \sim PSL_2(5)} &
 \mlns{ 10\,01\,01\,10\,01\,10_{\times 12}\\
        00\,01\,01\,01\,11\,01_{\times 30}\\
        01\,00\,00\,11\,10\,11_{\times 60} } &
\mlns{|\mathrm{Aut}|=2 \cdot 60\\ 
      \mbox{dual: }1^{12}}
 \\[7mm] \hline
%------------------------------------------
 \rbx{108a} &
 \mlns[l]{000\,011\,011\,\,110 _{\times 54\cdot 2}} &
 \mlns[c]{(012)(345)(678)\\(012)(876)(9ab)\\(630)(258)(9ab)\\ \sim SA_2(3)
} &
 \mlns{(06)(17)(28)(9b)\\ T_2: \\(06)(17)(28)(9a)} &
 \mlns{|\mathrm{Aut}|=2 \cdot 432\\ 
      \mbox{dual: }1^{12}\\ \mbox{Witt}}
 \\[7mm] \hline
 %------------------------------------------
 \rbx{108b} &
 \mlns[l]{0001\,0110\,1101_{\times 72}\\
           0011\,0011\,0110_{\times 36} } &
 \mlns[c]{(01)(23)(45)(67)(89)(ab)\\(123)(567)(9ab)\\(04)(15)(26)(37)\\(08)(19)(2a)(3b)\\ \sim A_4\times S_3} &
 \mlns{(01)(45)(89)} &
\mlns{|\mathrm{Aut}|=2 \cdot 144\\ 
      \mbox{dual: }1^{12}\\ \mbox{Witt}}
 \\[7mm] \hline
%------------------------------------------
 \rbx{110a} &
 \mlns[l]{000010011111_{\times 110}} &
 \mlns[c]{(0123456789a)\\(13954)(267a8)\\(0b)(1a)(25)(37)(48)(69)\\\sim PSL_2(11)} &
 \mlns{(0a)(19)(28)(37)(46)\\T_2:000010111011} &
 \mlns{|\mathrm{Aut}|=2 \cdot 1320\\ 
      \mbox{dual: }1^{12}}
 \\[7mm] \hline
%------------------------------------------
 \rbx{110b} &
 \mlns[l]{000010111011_{\times 110}} &
 \mlns[c]{(0123456789a)\\(13954)(267a8)\\(0b)(1a)(25)(37)(48)(69)\\\sim PSL_2(11)} &
 \mlns{000010011111_{\times 110} \\T_2:\\000011110011_{\times 110}} &
\mlns{|\mathrm{Aut}|=2 \cdot 660\\ 
      \mbox{dual: }1^{12}\\
      \mbox{no squares}}
 \\[7mm] \hline
%------------------------------------------
 \rbx{120a} &
 \mlns[l]{0011\,0101\,0101_{\times 24}\\0001\,0011\,1011_{\times 96}} &
 \mlns[c]{(0527)(1634)\\(048)(159)(26a)(37b)} &
 \mlns{(02)(46)(9b)} &
\mlns{|\mathrm{Aut}|=2 \cdot 192\\ 
      \mbox{dual: }1^{12}\\ \mbox{Witt}}
 \\[7mm] \hline
%------------------------------------------
 \rbx{120b} &
 \mlns[l]{01\,11\,10\,10\,10\,00_{\times 60}\\
          00\,00\,10\,01\,11\,11_{\times 60}} &
 \mlns[c]{(02468)(13579)\\(01)(23)(48)(59)(6a)(7b)\\ \sim PSL_2(5)} &
 \mlns{(01)(23)(45)(67)(89)(ab)} &
\mlns{|\mathrm{Aut}|=2 \cdot 120\\ 
      \mbox{dual: }1^{12}\\ \mbox{Witt}}
 \\[7mm] \hline
%------------------------------------------
 \rbx{132} &
 \mlns[l]{000001011111_{\times 132}} &
 \mlns[c]{(0123456789a)\\(13954)(267a8)\\(0b)(1a)(25)(37)(48)(69)\\ \sim PSL_2(11)} &
 \mlns{(0a)(19)(28)(37)(46)} &
\mlns{|\mathrm{Aut}|=2 \cdot 1320\\ 
      \mbox{dual: }1^{12}\\ \mbox{Witt}}
 \\[7mm] \hline
\end{longtable}

\subsection{Some additional results of the classification}\label{ss:res-plus}
\subsubsection{Small Witt design}
The possible differences of two small Witt designs S$(5,6,12)$ were classified in \cite{KraMes:74};
the results in this paragraph show the place for these differences in our classification, 
but do not establish new facts.
The trade of the maximum volume, $132$, consists of two S$(5,6,12)$. 
Only $7$ nonequivalent trades, with numbers
72b, 96a, 108a, 108b, 120a, 120b, 132,
can be represented as the difference pair
$(W_0\backslash W_1, W_1\backslash W_0)$ 
of two S$(5,6,12)$ $W_0$ and $W_1$.
This was established by an additional run of the algorithm
with $T_0$ being restricted by only the elements
of a fixed S$(5,6,12)$ (the search was fast enough without 
the isomorph rejection for partial solutions).
\begin{corollary}[\cite{KraMes:74}]\label{c:2witt} 
 Up to equivalence, there is only one pair 
 of disjoint S$(5,6,12)$.
\end{corollary}
\begin{proof}
Two disjoint S$(5,6,12)$ 
systems $W_0$ and $W_1$ form a trade 
$(W_0,W_1)$.
If it is not primary, then there is 
a trade 
$(V_0,V_1)$ such that $V_i\subset W_i$, $i=0$, $1$,
and $0 < |V_0| \le |W_0|/2 = 66$,
which is not possible as the minimum trade included in an  S$(5,6,12)$
has the volume $72$. 
Since a primary trade of volume $132$ is unique,
the statement follows.
\end{proof}

\subsubsection{$3$-way trades and no more}
Four of the trades, 72a, 108a, 110a, and 110b, can be continued 
to $3$-way trades $(T_0,T_1,T_2)$. It occurs that for given
$T_0$ and $T_1$, the choice of $T_2$ is unique;
it follows
that for the considered parameters, 
no primary trades can be continued to $k$-way trades for $k>3$.
The $3$-way trades from the trades 110a and 110b are the same; 
so, there are only three 
nonequivalent $3$-way trades 
obtained by continuing primary trades with considered parameters.
%(to be completed: it should be checked that a primary $3$-way trade
%cannot consist of $3$ nonprimary trades).

\subsubsection{Derived STS trades}
$87$ nonequivalent STS trades 
are derived from these extended $1$-perfect trades.
Among the STS trades of volume at most $9$,
only numbers 13 (of volume 8), 29, 30, 31, 33 (of volume 9), 
in the classification
 \cite[Table 3.4]{FGG:2004:trades} are not derived.

 Only the trades $32$ and $132$ are \emph{STS-uniform}, 
 that is, each has only one derived STS trade, up to equivalence.
Each of the three trades 72a, 110a, 110b has only two nonequivalent derived STS trades;
 the other $20$ trades have $4$ or more.

 All STS $3$-way trades of volume at most $9$
 (numbers 3, 8, 23, 25, and 28 in \cite[Table 3.4]{FGG:2004:trades})
 are derived from the $3$ found extended $1$-perfect $3$-way trades:
 number 3 is derived from the $3$-way trade of volume $72$;
 number 8, 23, 25 are derived from the $3$-way trade of volume $108$;
 number 28 is derived from the $3$-way trades of volumes $72$ and $108$.
 
\subsubsection{Other remarks}
It should be noted that all $25$ trades found are rather symmetric.
The trade of volume $98$ has the smallest automorphism group among all found trades.
Its symmetry group is isomorphic to the dihedral group of order $12$
and acts transitively on the coordinates.
However, with the growth of length, 
it is naturally to conjecture the existence of primary trades with small automorphism groups,
even consisting of the identity and the translation by the all-one vector only.

% The difference between nearest volumes of trades is only $2$ 
% ($\mbox{}=98-96 = 110-108$).
% Note that $1$ is not possible as by Corollary~\ref{c:antipod} the volumes are even.
% Moreover, this difference is realized on two transitive (consisting of one orbit, under the automorphism group) 
% trades 110a and 108a.

Trades 110a, 110b, and 132 consist of orbits of the same symmetry group (three orbits of cardinality $110$,
which form a $3$-way trade, and two orbits of cardinality $132$, Steiner sextuple systems S$(5,6,12)$),
the projective special group $\mathrm{PSL}_2(11)$, with the natural action on the $12$ coordinates.
The other orbits of this group are not connected with trades.

Trade 110b is the only trade whose graph (i.e., the subgraph of \HammingGraphalvedCube{12} induced by $T_0\cup T_1$) has no cycles of length $4$ (squares).

\section{Construction of extended $1$-perfect trades}\label{s:constr}

In this section, we show how extended $1$-perfect trades or $k$-way 
trades of small length can be used to construct trades of larger length. 
In particular, if the length of a trade is not a power of two, 
it obviously cannot be embedded into a pair of extended $1$-perfect codes of the same length,
but the question if it can be embedded after lengthening by the construction below is considerably more difficult.

The construction uses latin trades, 
whose construction is not discussed here;
however, there is a simple example 
of a latin trade $(L_0,L_1)$ that can be used in the construction: 
$L_j$ consists of all words of length $m$ over the alphabet $\{0,\ldots,q-1\}$ with the sum of all coordinates being $j$ modulo $q$.
The simplest case that works is $q=2$, and the simplest extended $1$-perfect trade that can be used in the construction has length $2$: $(\{00\},\{11\})$; 
the corresponding partial case is described in Lemma~\ref{p:constr}.
The construction below is a variant of the product construction of extended $1$-perfect codes suggested in \cite{Phelps84} 
and also inherit ideas of the generalized concatenation construction for error-correcting codes from \cite{Zin1976:GCC}.
The proof is straightforward and omitted here. 
\begin{proposition}\label{p:constr}
Let $M=(M_0,M_1)$ be a latin trade in \HammingGraph{m,q}, and let for every
$i$ from $0$ to $m-1$, only the symbols $0$, \ldots, $q_i-1$, $q_i\le q$, 
occur in the $i$th position of the words of $M$.
For each $i$ from $0$ to $m-1$, 
let $T^{(i)}=(T^{(i)}_0,\ldots,T^{(i)}_{q_i-1})$ be
an extended $1$-perfect $q_i$-way trade in \HammingGraphalvedCube{n_i}. 
Then the pair $(T_0,T_1)$, where
$$ T_j = \{ (c_0 \ldots c_{m-1}) \,|\, c_i\in T^{(i)}_{b_i}, (b_0\ldots b_{m-1})\in M_j\},$$
is an extended $1$-perfect trade in \HammingGraphalvedCube{n_0+\ldots+n_{m-1}}. Moreover, 

(1) if the trade $M$ is primary and for every $i\in \{0,\ldots,m-1\}$ and every different $j$, $j'$ from $ \{0,\ldots,q_i-1\}$, 
the trade $(T^{(i)}_j,T^{(i)}_{j'})$ is primary, then the trade $(T_0,T_1)$ is primary too;

(2) for every $i\in \{0,\ldots,m-1\}$, the word 
$0^{n_0}...0^{n_{i-1}}1^{n_i}0^{n_{i+1}}...0^{n_{m-1}}$ is dual to $(T_0,T_1)$.
\end{proposition}
Similarly, if we replace the latin trade by a latin $k$-way trade, 
then we obtain an extended $1$-perfect $k$-way trade.

A special case of $1$-perfect trade mates, 
subsets of $1$-perfect binary codes called $i$-components, 
can be constructed using another approach, see e.g. \cite{Sol:2001i-comp}. 
In particular, for each length of form $2^m-1$, 
there are primary $1$-perfect trades of volume $2^m-m-2$,
i.e., half of the cardinality of a $1$-perfect code 
(readily, the same is true for extended $1$-perfect trades and codes of length $2^m$).

\section{Concluding remarks}\label{s:end}

We presented some classification results on extended $1$-perfect trades, % and unitrades, 
obtained by computer search, for small parameters.
In the conclusion, we would like to emphasize a connection of some of found trades with optimal codes.

The optimal distance-$4$ binary code of length $10$
(it has $40$ codewords,
is known as the Best code \cite{Best80},
and it is unique up to equivalence \cite{LitVar:Best})
form an extended $1$-perfect trade with its complement.

The optimal distance-$4$ constant-weight binary codes of lengths $n=10$ and $n=12$ and weight $n/2$
(the cardinality is $36$ and $132$, respectively) are extended $1$-perfect trade mates.
% (in the first case, the other mate is the complement of the code,
% in the second case, the uniqueness of the other mate is a nontrivial result of computing).

% The found largest unitrades of length $10$ can be considered as two-fold packings of radius-$1$ balls, 
% which are very close to the sphere-packing bound and conjectured to be optimal.
% These unitrades, in contrast to the bitrades, cannot be represented as eigenfunctions, 
% but we found that they can be represented as equitable partitions with certain intersection matrix.

So, we found new illustrations to the fact that some good codes can be represented 
as algebraic-combinatorial objects like eigenfunctions  
(classical examples of such codes are the perfect codes; 
less obvious examples are, e.g., the binary $(n=2^k-3, 2^{n-k} , 3)$ and $(n=2^k-4, 2^{n-k} , 3)$ codes \cite{Kro:2m-3},  \cite{Kro:2m-4}).
It would be quite interesting to continue classification of the extended $1$-perfect trades in small 
dimensions and try to find more connections with good codes. 
However, the possibilities of the considered algorithm seem to be exhausted.
% Using the algorithm described, it is hardly possible to classify the bitrades of a larger length
The attempts to start it with larger parameters did not allow even to evaluate the time needed to complete the search.
The number of the objects we search grows double-exponentially ($2^{2^{O(n)}}$), 
and the complexity of any algorithm finds a physical limit rather fast.
Some hope to find interesting examples in larger lengths by computer search
is related with the search of objects with some restrictions, 
for example, with a prescribed automorphism group.

\section*{Acknowledgements}
The author thanks the referees for their work with the manuscript,
which enable to improve the quality of the text.
The research was carried out at the Sobolev Institute of Mathematics 
at the expense of the Russian Science Foundation (Grant 14-11-00555).

%\section*{References}

%\bibliographystyle{plain}
%\bibliography{../../k}

\begin{thebibliography}{10}
\expandafter\ifx\csname url\endcsname\relax
  \def\url#1{\texttt{#1}}\fi
\expandafter\ifx\csname urlprefix\endcsname\relax\def\urlprefix{URL }\fi
\expandafter\ifx\csname href\endcsname\relax
  \def\href#1#2{#2} \def\path#1{#1}\fi
  
\bibitem{AvgHedSol:class}
S.~V. Avgustinovich, O.~Heden, F.~I. Solov'eva, The classification of some
  perfect codes, \href{http://link.springer.com/journal/10623}{Des. Codes
  Cryptography} 31~(3) (2004) 313--318.
\newblock \href {http://dx.doi.org/10.1023/B:DESI.0000015891.01562.c1}
  {\path{doi:10.1023/B:DESI.0000015891.01562.c1}}

\bibitem{AvgKro:embed}
S.~V. Avgustinovich, D.~S. Krotov, Embedding in a perfect code,
  \href{http://onlinelibrary.wiley.com/journal/10.1002/(ISSN)1520-6610}{J.
  Comb. Des.} 17~(5) (2009) 419--423.
\newblock \href {http://dx.doi.org/10.1002/jcd.20207}
  {\path{doi:10.1002/jcd.20207}}

\bibitem{Best80}
M.~R. Best, Binary codes with a minimum distance of four,
  \href{http://ieeexplore.ieee.org/xpl/RecentIssue.jsp?punumber=18}{IEEE Trans.
  Inf. Theory} 26~(6) (1980) 738--742.
\newblock \href {http://dx.doi.org/10.1109/TIT.1980.1056269}
  {\path{doi:10.1109/TIT.1980.1056269}}

\bibitem{Brouwer}
A.~E. Brouwer, A.~M. Cohen, A.~Neumaier, Distance-Regular Graphs,
  Springer-Verlag, Berlin, 1989.
\newblock \href {http://dx.doi.org/10.1007/978-3-642-74341-2}
  {\path{doi:10.1007/978-3-642-74341-2}}

\bibitem{Carmichael:31}
R.~D. Carmichael, Configurations of rank two, Am. J. Math. 53~(1) (1931)
  217--240.
\newblock \href {http://dx.doi.org/10.2307/2370885}
  {\path{doi:10.2307/2370885}}

\bibitem{Cav:rev}
N.~J. Cavenagh, The theory and application of latin bitrades: A survey,
  \href{http://www.springer.com/mathematics/journal/12175}{Math. Slovaca}
  58~(6) (2008) 691--718.
\newblock \href {http://dx.doi.org/10.2478/s12175-008-0103-2}
  {\path{doi:10.2478/s12175-008-0103-2}}

\bibitem{FGG:2004:trades}
A.~D. Forbes, M.~J. Grannell, T.~S. Griggs,
  \href{http://ajc.maths.uq.edu.au/pdf/29/ajc_v29_p075.pdf}{Configurations and
  trades in {S}teiner triple systems},
  \href{http://ajc.maths.uq.edu.au}{Australas. J. Comb.} 29 (2004) 75--84.
\urlprefix\url{http://ajc.maths.uq.edu.au/pdf/29/ajc_v29_p075.pdf}

\bibitem{HedKho:trades}
A.~S. Hedayat, G.~B. Khosrovshahi, Trades, in: C.~J. Colbourn, J.~H. Dinitz
  (Eds.), Handbook of Combinatorial Designs, 2nd Edition, Discrete Mathematics
  and Its Applications, Chapman \& Hall/CRC, Boca Raton, London, New York,
  2006, pp. 644--648.

\bibitem{Hed:2008:survey}
O.~Heden, A survey of perfect codes,
  \href{http://aimsciences.org/journals/amc/}{Adv. Math. Commun.} 2~(2) (2008)
  223--247.
\newblock \href {http://dx.doi.org/10.3934/amc.2008.2.223}
  {\path{doi:10.3934/amc.2008.2.223}}

\bibitem{HK:q-ary}
O.~Heden, D.~S. Krotov, On the structure of non-full-rank perfect $q$-ary
  codes, \href{http://aimsciences.org/journals/amc/}{Adv. Math. Commun.} 5~(2)
  (2011) 149--156.
\newblock \href {http://dx.doi.org/10.3934/amc.2011.5.149}
  {\path{doi:10.3934/amc.2011.5.149}}

\bibitem{Hergert:85}
F.~Hergert, Algebraische methoden f\"ur nichtlineare codes, Dissertation,
  Technischen Hochschule Darmstadt (1985).

\bibitem{KO:alg}
P.~Kaski, P.~R.~J. {\"O}sterg{\aa}rd, Classification Algorithms for Codes and
  Designs, Vol.~15 of Algorithms Comput. Math., Springer, Berlin, 2006.
\newblock \href {http://dx.doi.org/10.1007/3-540-28991-7}
  {\path{doi:10.1007/3-540-28991-7}}

\bibitem{KraMes:74}
E.~S. Kramer, D.~M. Mesler, Intersections among {S}teiner systems,
  \href{http://www.sciencedirect.com/science/journal/00973165}{J. Comb. Theory,
  Ser.~A} 16~(3) (1974) 273--285.
\newblock \href {http://dx.doi.org/10.1016/0097-3165(74)90054-5}
  {\path{doi:10.1016/0097-3165(74)90054-5}}
  
\bibitem{Kro:2m-3}
D.~S. Krotov, On the binary codes with parameters of doubly-shortened
  $1$-perfect codes, \href{http://link.springer.com/journal/10623}{Des. Codes
  Cryptography} 57~(2) (2010) 181--194.
\newblock \href {http://dx.doi.org/10.1007/s10623-009-9360-5}
  {\path{doi:10.1007/s10623-009-9360-5}}

\bibitem{Kro:2m-4}
D.~S. Krotov, On the binary codes with parameters of triply-shortened
  $1$-perfect codes, \href{http://link.springer.com/journal/10623}{Des. Codes
  Cryptography} 64~(3) (2012) 275--283.
\newblock \href {http://dx.doi.org/10.1007/s10623-011-9574-1}
  {\path{doi:10.1007/s10623-011-9574-1}}

\bibitem{KMP:16:trades}
D.~S. Krotov, I.~Y. Mogilnykh, V.~N. Potapov, To the theory of $q$-ary
  {S}teiner and other-type trades,
  \href{http://www.sciencedirect.com/science/journal/0012365X}{Discrete Math.}
  339~(3) (2016) 1150--1157.
\newblock \href {http://dx.doi.org/10.1016/j.disc.2015.11.002}
  {\path{doi:10.1016/j.disc.2015.11.002}}

\bibitem{LitVar:Best}
S.~Litsyn, A.~Vardy, The uniqueness of the {B}est code,
  \href{http://ieeexplore.ieee.org/xpl/RecentIssue.jsp?punumber=18}{IEEE Trans.
  Inf. Theory} 40~(5) (1994) 1693--1698.
\newblock \href {http://dx.doi.org/10.1109/18.333896}
  {\path{doi:10.1109/18.333896}}

\bibitem{MWS}
F.~J. MacWilliams, N.~J.~A. Sloane, The Theory of Error-Correcting Codes,
  Amsterdam, Netherlands: North Holland, 1977.

\bibitem{nauty2014}
B.~D. McKay, A.~Piperno, Practical graph isomorphism, {II}, J. Symb. Comput. 60
  (2014) 94--112.
\newblock \href {http://dx.doi.org/10.1016/j.jsc.2013.09.003}
  {\path{doi:10.1016/j.jsc.2013.09.003}}

\bibitem{Ost:2012:switching}
P.~R.~J. {\"O}sterg{\aa}rd, Switching codes and designs,
  \href{http://www.sciencedirect.com/science/journal/0012365X}{Discrete Math.}
  312~(3) (2012) 621--632.
\newblock \href {http://dx.doi.org/10.1016/j.disc.2011.05.016}
  {\path{doi:10.1016/j.disc.2011.05.016}}

\bibitem{OstPot2007}
P.~R.~J. {\"O}sterg{\aa}rd, O.~Pottonen, There exist {S}teiner triple systems
  of order $15$ that do not occur in a perfect binary one-error-correcting
  code,
  \href{http://onlinelibrary.wiley.com/journal/10.1002/(ISSN)1520-6610}{J.
  Comb. Des.} 15~(6) (2007) 465--468.
\newblock \href {http://dx.doi.org/10.1002/jcd.20122}
  {\path{doi:10.1002/jcd.20122}}

\bibitem{Phelps84}
K.~T. Phelps, A general product construction for error correcting codes,
  \href{http://epubs.siam.org/journal/sjamdu}{SIAM J. Algebraic Discrete
  Methods} 5~(2) (1984) 224--228.
\newblock \href {http://dx.doi.org/10.1137/0605023}
  {\path{doi:10.1137/0605023}}

\bibitem{Pot12:spectra}
V.~N. Potapov, Cardinality spectra of components of correlation immune
  functions, bent functions, perfect colorings, and codes,
  \href{http://link.springer.com/journal/11122}{Probl. Inf. Transm.} 48~(1)
  (2012) 46--54, translated from
  \href{http://www.mathnet.ru/php/journal.phtml?jrnid=ppi\&option_lang=eng}{Probl.
  Peredachi Inf.} 48~(1) (2012) 54--63.
\newblock \href {http://dx.doi.org/10.1134/S003294601201005X}
  {\path{doi:10.1134/S003294601201005X}}

\bibitem{Potapov:2013:trade}
V.~N. Potapov, Multidimensional {L}atin bitrades,
  \href{http://link.springer.com/journal/11202}{Sib. Math. J.} 54~(2) (2013)
  317--324, translated from
  \href{http://www.mathnet.ru/php/journal.phtml?jrnid=smj\&option_lang=eng}{Sib.
  Mat. Zh.} 54~(2) (2013) 407--416.
\newblock \href {http://dx.doi.org/10.1134/S0037446613020146}
  {\path{doi:10.1134/S0037446613020146}}

\bibitem{Rom:survey}
A.~M. Romanov, Survey of the methods for constructing nonlinear perfect binary
  codes, \href{http://link.springer.com/journal/volumesAndIssues/11754}{J.
  Appl. Ind. Math.} 2~(2) (2008) 252--269, translated from 
  \href{http://www.mathnet.ru/php/journal.phtml?jrnid=da\&option_lang=eng}{Diskretn.
  Anal. Issled. Oper.} Ser.~1  13~(4) (2006) 60--88.
\newblock \href {http://dx.doi.org/10.1134/S1990478908020105}
  {\path{doi:10.1134/S1990478908020105}}

\bibitem{Sol:switchings}
F.~I. Solov'eva, Switchings and perfect codes, in: I.~Alth{\"o}fer, N.~Cai,
  G.~Dueck, L.~Khachatrian, M.~Pinsker, G.~Sarkozy, I.~Wegener, Z.~Zhang
  (Eds.), Numbers, Information and Complexity, Kluwer Academic Publisher, 2000,
  pp. 311--314.

\bibitem{Sol:2001i-comp}
F.~I. Solov'eva, Structure of $i$-components of perfect binary codes,
  \href{http://www.sciencedirect.com/science/journal/0166218X}{Discrete Appl.
  Math.} 111~(1-2) (2001) 189--197.
\newblock \href {http://dx.doi.org/10.1016/S0166-218X(00)00352-8}
  {\path{doi:10.1016/S0166-218X(00)00352-8}}

\bibitem{sage}
W.~A. Stein, et~al., \href{http://www.sagemath.org}{{S}age {M}athematics
  {S}oftware ({V}ersion 6.9)} (2015).
\urlprefix\url{http://www.sagemath.org}

\bibitem{Street:trades}
A.~P. Street, Trades and defining sets, in: C.~J. Colbourn, J.~H. Dinitz
  (Eds.), Handbook of Combinatorial Designs, Discrete Mathematics and Its
  Applications, CRC press, Boca Raton, New York, London, Tokio, 1996, pp.
  474--478.

\bibitem{VAK:2008}
Y.~L. Vasil'ev, S.~V. Avgustinovich, D.~S. Krotov, On mobile sets in the binary
  hypercube,
  \href{http://www.mathnet.ru/php/journal.phtml?jrnid=da\&option_lang=eng}{Diskretn.
  Anal. Issled. Oper.} 15~(3) (2008) 11--21, in Russian, translated in
  \url{http://arxiv.org/abs/0802.0003}

\bibitem{VorKro:2014en}
V.~K. Vorob'ev, D.~S. Krotov, Bounds for the size of a minimal $1$-perfect
  bitrade in a {H}amming graph,
  \href{http://link.springer.com/journal/volumesAndIssues/11754}{J. Appl. Ind.
  Math.} 9~(1) (2015) 141--146, translated
  from
  \href{http://www.mathnet.ru/php/journal.phtml?jrnid=da\&option_lang=eng}{Diskretn.
  Anal. Issled. Oper.} 21~(6) (2014) 3--10.
\newblock \href {http://dx.doi.org/10.1134/S1990478915010159}
  {\path{doi:10.1134/S1990478915010159}}

\bibitem{Witt:37}
V.~E. Witt, {\"U}ber {S}teinersche systeme, Abh. Math. Semin. Univ. Hamb.
  12~(1) (1937) 265--275.
\newblock \href {http://dx.doi.org/10.1007/BF02948948}
  {\path{doi:10.1007/BF02948948}}

\bibitem{Zin1976:GCC}
V.~A. Zinoviev, Generalized concatenated codes,
  \href{http://link.springer.com/journal/11122}{Probl. Inf. Transm.} 12~(1)
  (1976) 2--9, translated from  Probl. Peredachi Inf. 12~(1) (1976) 5--15.

\end{thebibliography}
%\end{document}

\providecommand\href[2]{#2} \providecommand\url[1]{\href{#1}{#1}}

\end{document}